\documentclass[12pt]{amsart}
\usepackage{times}
\usepackage{amscd,verbatim,amssymb}
\usepackage[all]{xy}
\usepackage{color}
\usepackage[T1]{fontenc}
\usepackage[colorlinks,linkcolor=blue,citecolor=blue,urlcolor=red]{hyperref}

\newcommand{\A}{\mathbf{A}}

\newcommand{\C}{\mathbf{C}}
\newcommand{\F}{\mathbf{F}}
\newcommand{\G}{\mathbb{G}}

\renewcommand{\P}{\mathbf{P}}
\newcommand{\Q}{\mathbf{Q}}

\newcommand{\Z}{\mathbf{Z}}

\newcommand{\sA}{\mathcal{A}}

\newcommand{\sC}{\mathcal{C}}

\newcommand{\sF}{\mathcal{F}}

\newcommand{\sL}{\mathcal{L}}

\newcommand{\sN}{\mathcal{N}}
\newcommand{\sO}{\mathcal{O}}
\newcommand{\cO}{{\mathcal O}}

\newcommand{\sX}{\mathcal{X}}
\newcommand{\sY}{\mathcal{Y}}
\newcommand{\sZ}{\mathcal{Z}}

\newcommand{\fm}{\mathfrak{m}}

\newcommand{\et}{{\operatorname{\acute{e}t}}}

\newcommand{\Mor}{\operatorname{Mor}}

\newcommand{\Alb}{\operatorname{Alb}}
\newcommand{\Ab}{\operatorname{\mathbf{Ab}}}

\newcommand{\Tr}{\operatorname{Tr}}

\newcommand{\alg}{{\operatorname{alg}}}

\newcommand{\num}{{\operatorname{num}}}

\newcommand{\cont}{{\operatorname{cont}}}

\newcommand{\car}{\operatorname{char}}
\newcommand{\Spec}{\operatorname{Spec}}

\newcommand{\Pic}{\operatorname{Pic}}

\newcommand{\LN}{\operatorname{LN}}
\newcommand{\Tor}{\operatorname{Tor}}
\newcommand{\Hom}{\operatorname{Hom}}

\newcommand{\Ext}{\operatorname{Ext}}
\newcommand{\CHC}{\operatorname{CH\mathcal{C}}}
\newcommand{\codim}{\operatorname{codim}}

\newcommand{\tors}{{\operatorname{tors}}}

\newcommand{\cl}{{\operatorname{cl}}}
\newcommand{\AJ}{{\operatorname{AJ}}}
\newcommand{\sm}{{\operatorname{sm}}}
\newcommand{\Coker}{\operatorname{Coker}}
\newcommand{\Ker}{\operatorname{Ker}}

\newcommand{\rk}{\operatorname{rk}}
\newcommand{\IM}{\operatorname{Im}}

\newcommand{\Supp}{\operatorname{Supp}}
\newcommand{\Ann}{\operatorname{Ann}}
\newcommand{\trdeg}{\operatorname{trdeg}}

\renewcommand{\phi}{\varphi}
\renewcommand{\epsilon}{\varepsilon}

\newcommand{\inj}{\hookrightarrow}

\newcommand{\by}{\xrightarrow}
\newcommand{\yb}{\xleftarrow}
\newcommand{\tto}{\dashrightarrow}
\newcommand{\iso}{\by{\sim}}
\newcommand{\osi}{\yb{\sim}}
\renewcommand{\lim}{\varprojlim}
\newcommand{\colim}{\varinjlim}

\newcounter{spec}
\newenvironment{thlist}{\begin{list}{\rm{(\roman{spec})}}%
{\usecounter{spec}\labelwidth=20pt\itemindent=0pt\labelsep=10pt}}%
{\end{list}}%

\DeclareFontFamily{U}{wncy}{}
\DeclareFontShape{U}{wncy}{m}{n}{%
<5>wncyr5%
<6>wncyr6%
<7>wncyr7%
<8>wncyr8%
<9>wncyr9%
<10>wncyr10%
<11>wncyr10%
<12>wncyr6%
<14>wncyr7%
<17>wncyr8%
<20>wncyr10%
<25>wncyr10}{}
\DeclareMathAlphabet{\cyr}{U}{wncy}{m}{n}

\newcommand{\cB}{\cyr{B}}

\newtheorem{thm}{Theorem}[section]
\newtheorem{conj}[thm]{Conjecture}

\newtheorem{lemma}[thm]{Lemma}
\newtheorem{claim}[thm]{Claim}
\newtheorem{prop}[thm]{Proposition}
\newtheorem{cor}[thm]{Corollary}

\theoremstyle{definition}
\newtheorem{defn}[thm]{Definition}
\theoremstyle{remark}
\newtheorem{rque}[thm]{Remark}
\newtheorem{rques}[thm]{Remarks}
\newtheorem{ex}[thm]{Example}

\newtheorem{qn}[thm]{Question}

\begin{document}
\title{Refined height pairing}
\author[Bruno Kahn]{Bruno Kahn}
\address{CNRS, Sorbonne Université and Université Paris Cité, IMJ-PRG\\ Case 247\\4 place
Jussieu\\75252 Paris Cedex 05\\France}
\email{bruno.kahn@imj-prg.fr}
\address{
Univ. Bordeaux, CNRS,  Institut de mathématiques de Bordeaux, UMR 5251,
F-33405 Talence, 
 France}
\email{Qing.Liu@math.u-bordeaux.fr}
\date{Dec. 6, 2023}
\keywords{Intersection theory, alteration theory, category theory}
\subjclass[2010]{14C17,14E15}
\maketitle

\hfill With an appendix by Qing Liu

\begin{abstract}For a $d$-dimensional regular proper variety $X$ over the function field of a smooth variety $B$ over a field $k$ and for $i\ge 0$, we define a subgroup $CH^i(X)^{(0)}$ of $CH^i(X)$ and construct a ``refined height pairing''
\[CH^i(X)^{(0)}\times CH^{d+1-i}(X)^{(0)}\to CH^1(B)\]
 in the category of abelian groups up to isogeny. For $i=1,d$, $CH^i(X)^{(0)}$ is the group of cycles numerically equivalent to $0$. This pairing relates to pairings  defined by P. Schneider and A. Beilinson if $B$ is a curve, to a refined height defined by L. Moret-Bailly when $X$ is an abelian variety, and to a pairing with values in $H^2(B_{\bar k},\Q_l(1))$ defined by D. Rössler and T. Szamuely in general. We study it in detail when $i=1$.
 \end{abstract}

\tableofcontents

\section*{Introduction} 
Let $X$ be a regular proper  (for example, smooth projective) variety of dimension $d$ over a field $K$, finitely generated of transcendence degree $\delta$ over a subfield $k$. Suppose given a smooth (separated) $k$-scheme of finite type $B$, with function field $K$. For $i\in [0,d]$, write $CH^i(X)$ for the $i$-th Chow group of $X$. In this paper, we define a subgroup $CH^i(X)^{[0]}$ and a ``refined height pairing''
\begin{equation}\label{eq0}
CH^i(X)^{[0]}\times CH^{d+1-i}(X)^{[0]}\to CH^1(B)
\end{equation}
in the category $\Ab\otimes \Z[1/p]$ of abelian groups up to $p$-isogeny: this category 
is recalled in \S \ref{s3.2}. Here $p$ is the exponential characteristic of $k$, so nothing is inverted in characteristic $0$; the only reason to invert it in nonzero characteristic is a lack of resolution of singularities, see \S \ref{s4.1}. 

If $B$ is a smooth projective curve and we compose \eqref{eq0} with the degree map, we get a $\Z[1/p]$-valued pairing (with values in $p^{-s} \Z$ for some integer $s\ge 0$), which relates to the one constructed by Beilinson in \cite[\S 1]{beilinson}. In \cite[p. 5]{beilinson}, Beilinson asked what happens when $\trdeg(K/k)>1$: \eqref{eq0} gives one answer to this question.
\bigskip

The quotient $CH^i(X)/CH^i(X)^{[0]}$ is finitely generated. When we vary $(X,i)$, $CH^i(X)^{[0]}$ defines an adequate equivalence relation for smooth projective $K$-varieties, which a priori depends on the choice of $B$.  Its saturation $CH^i(X)^{(0)}$ lies between the subgroups $CH^i_\alg(X)$ and  $CH^i_\num(X)$ of  algebraically and numerically trivial cycles, hence equals $CH^i_\num(X)$ when $i=1,d$. We conjecture that this holds for all $i$, and prove it in further special cases (Theorem \ref{t5.1} (ii)). One can show that it would follow in general from the Tate conjecture, or the Hodge conjecture in characteristic $0$, for cycles of codimension $<i$, although we don't include a proof here. 
More generally, one might hope that Lemma \ref{l1} below induces pairings in $\Ab\otimes \Q$
\[F^nCH^i(X)\times F^nCH^{d+n-i}(X)\to CH^n(B), \quad i\ge 0\]
where $F^*CH^*(X)$ is the conjectural Bloch-Beilinson--Murre filtration \cite{jannsen2}, the case $n=0$ (resp. $1$) being the intersection pairing (resp. \eqref{eq0}).
\bigskip

Works following Néron's seminal paper \cite{neron} have much relied on $l$-adic cohomology to analyse or define height pairings (because of the cohomological definition of Hasse-Weil $L$-functions): for $\delta=1$, this is the case in Schneider \cite{schneider} ($i=1$, $X$ an abelian variety), Bloch \cite{bloch} and Beilinson \cite{beilinson}. This is also the case in the work of Damian R\"ossler and Tam\'as Szamuely \cite{RS} which is the direct inspiration of this one: they construct a pairing
\begin{equation}\label{eq01}
CH^i_l(X)\times CH^{d+1-i}_l(X)\to H^2_\et(B_{\bar k},\Q_l(1))
\end{equation}
where $l$ is a prime number invertible in $k$ and $CH^i_l(X)$ denotes cycles homologically equivalent to $0$ with respect to $l$-adic cohomology. By contrast, our approach here is completely cycle-theoretic and very close in spirit to  Moret-Bailly's geometric height from \cite[Ch. III, Déf. 3.2]{mb}; it relies on Fulton's marvellous theory of Gysin maps  \cite[Ch. 6 and 8]{fulton}. This gives a different flavour to the definitions because numerical and homological equivalence have rather opposite functoriality under specialisation, as described in detail by Grothendieck in \cite[7.9 and 7.13]{groth}. See Remark \ref{r2.4}.
\bigskip

Comparing various definitions of height pairings is a highly nontrivial issue, which is solved only in a few cases: for example, as far as I know those defined by Bloch and Beilinson in \cite{bloch} and \cite{beilinson} have still not been checked to agree. In \cite{schneider}, Schneider compares an $l$-adic height pairing (loc. cit., p. 298) with the Néron-Tate height by comparing each to an intermediate Yoneda pairing (loc. cit., p. 502)
\begin{equation}\label{eq0.1}
H^0(B,\sA^0)\times \Ext^1_B(\sA^0,\G_m)\to CH^1(B)
\end{equation}
where $\sA^0$ is the connected component of the identity of the Néron model $\sA$ of the abelian variety $A$ (=$X$ here). 

In Proposition \ref{r2.5}, I show that \eqref{eq0} and \eqref{eq01} are compatible (at least in characteristic $0$) on a common subgroup $CH^*_{\cB,l}(X)$ of  $CH^*_l(X)$ and $CH^*(X)^{[0]}$ via the cycle class map $\Pic(B)\to H^2_\et(B_{\bar k},\Q_l(1))$: this is what Rössler and Szamuely had checked in the special case where $X/K$ has a \emph{smooth} model, by using a variant of Proposition \ref{p1} here \cite[Prop. 6.1]{RS}. In Theorem \ref{t5.2}, I show that \eqref{eq0} is the opposite of Silverman's refined height pairing of \cite[Th. III.9.5 (b)]{silverman} in the classical case of an elliptic curve $X$ over the function field of a smooth projective curve $B$ over an algebraically closed field $k$. 

Another case where a compatibility should not be hard to show is that of \cite{mb}.
\bigskip

Note that  \eqref{eq0} is finer than \eqref{eq01} inasmuch as it takes homologically trivial cycles on $B$ into account. This extra structure is presumably  arithmetically significant; it is studied in Subsection \ref{s5.1} in the case $d=1$, $B$ projective.
\bigskip

It may seem disturbing that \eqref{eq0} is essentially integral, while the classical height pairing is usually rational: this may be ``explained'' by \eqref{eq0.1} which is integral but takes values on the subgroup of finite index $\sA^0(B)\subseteq A(K)$. In this spirit, I show in Remark \ref{r5.1} a) that in the elliptic curve case mentioned above, $CH^1(X)^{[0]}$ contains $\sN^0(B)$ as a subgroup of finite index, where $\sN^0$ is the identity component of the Néron model of $X$.
\bigskip 

The \emph{raison d'être} of \cite{beilinson} and  \cite{bloch} was to  refine the conjectures of Tate on the orders of poles of zeta functions at integers \cite{tate}, by describing special values at these integers,  when $K$ is a global field. Thus one might like to extend \eqref{eq0} to the case where $B$ is regular and flat over $\Z$. I consider  this as beyond the scope of this article for two reasons:

\begin{itemize}
\item The present method fails in this case even if one is given a regular projective model $f:\sX\to B$ of $X$, because Fulton's techniques do not define an intersection product on $\sX$, except  when $\delta=1$ and $f$ is smooth \cite[p. 397]{fulton}. One does get an intersection product with $\Q$ coefficients, by using either $K$-theory as in Gillet-Soulé \cite[8.3]{gs}, or  alterations and deformation to the normal cone as in Andreas Weber's thesis \cite[Cor. 4.2.3 and Th. 4.3.3]{weber};  it is possible that the present approach may be adapted by using one of these products.
\item However, the main point in characteristic $0$ is to involve archi\-med\-ean places to get a complete height pairing whose determinant has a chance to describe the special values as mentioned above: this is what was done successfully in \cite{bloch} and \cite{beilinson} when $\delta=1$. In higher dimensions, one probably would have to use something like Arakelov intersection theory (see \cite[Conj. 7.1]{RS} for a conjectural statement).
\end{itemize}

I leave these issues to the interested readers. Rather, I hope to show here that height pairings in the style of \eqref{eq0} also raise interesting geometric questions. These are discussed in Section \ref{s5}, which is closely related to \cite[Question 7.6]{surfoncurve}.

\subsection*{Contents} Up to Subsection \ref{s3.1}, we assume $k$ perfect; this assumption is removed in the said subsection. In Definition \ref{r2.1}, we introduce subgroups $CH^i(\sX)^0$ of \emph{admissible cycles} in the Chow groups of a $k$-model $f:\sX\to B$ of $f':X\to \Spec K$, with $\sX$ smooth; when $B$ is projective, $CH^i(\sX)^0$ contains numerically trivial cycles (Proposition \ref{p2.4}) and in general it contains locally homologically trivial cycles in the sense of Beilinson \cite[1.2]{beilinson}  (Proposition \ref{p2.1}). From the intersection pairing on $\sX$, pushed forward to $CH^1(B)$, we then get thanks to Proposition \ref{p1} a height pairing $\langle, \rangle_f$ defined on the groups $CH^i(X)^0_f:=\IM(CH^i(\sX)^0\to CH^i(X))$ \eqref{eq4}. This is a pairing  of genuine abelian groups. We prove in Propositions \ref{l6} and \ref{r4} that  the $CH^i(X)^0_f$ and $\langle, \rangle_f$ are independent of $f$ and compatible with the action of correspondences, and  in Proposition \ref{p5} that they behave well with respect to base change. The group $CH^i(X)/CH^i(X)^0$ is finitely generated (Proposition \ref{p2.2}). 

If we are in characteristic $0$, the construction is finished since $X$ always admits a smooth model by resolution of singularities (Proposition \ref{p3.1}). In characteristic $p>0$, there turns out to be quite a bit of work to get a pairing in general after suitably inverting $p$, by using Gabber's refinement of de Jong's theorem: the general height pairing \eqref{eq9} is defined in Theorem \ref{p3}; as said above, it makes sense in the category $\Ab\otimes \Z[1/p]$. Functoriality and base change extend to this pairing (ibid.)

In Section \ref{s4}, we investigate Conjecture \ref{c1}: $CH^i(X)^{[0]}$ is of finite index in $CH^i_\num(X)$, the group of cycles numerically equivalent to $0$ (the inclusion is always true by Lemma \ref{l4.3} d)); we prove it for $i=1,d$ in Proposition \ref{p2} b) (see Theorem \ref{t5.1} (ii) for other cases). In \S \ref{s5.3}, we also relate \eqref{eq0} to the classical Néron-Tate height pairing in the case where $X$ is an elliptic curve and $B$ is a smooth projective curve.

In Section \ref{s5}, we study the height pairing \eqref{eq4} in the basic case $i=1$. If $B$ is projective, it leads to a coarser pairing \eqref{eq5.2} between the Lang-Néron groups $\LN(\Pic^0_X,K/k)$ and $\LN(\Alb_X,K/k)$ with values in $N^1(B)$, codimension $1$ cycles modulo numerical equivalence (Theorem \ref{t2}). When $\delta=1$, a version of this pairing involving an ample divisor is negative definite (Theorem \ref{t3}): one should compare this with a result of Shioda when $d=1$ \cite{shioda}. See also Theorem \ref{t3} for a conjectural statement when $\delta>1$. We finally get an intriguing homomorphism from $\LN(\Pic^0_X,K/k)$ to homomorphisms between certain abelian varieties in \eqref{eq5.4}.

\subsection*{Acknowlegdements} What triggered me to start this research was a talk by Tam\'as Szamuely at the June 2019 Oberwolfach workshop on algebraic $K$-theory, where he explained a preliminary version of \cite{RS}; I thank him and Damian Rössler for several exchanges during the preparation of this work. Part of this theory was developed while I was visiting Jilali Assim in Meknès in March 2020; I would like to thank him and Université Moulay Ismail for their hospitality and excellent working conditions. 
I thank Qing Liu for accepting to write the appendix and for helping me with references in EGA, Marc Hindry for a  discussion around Remark \ref{r5.1} a) and Tamás Szamuely for his help with the end of the proof of Proposition \ref{r2.5}.

It is common to thank the referees for helpful remarks. Here, I would like to stress my appreciation for exceptionally lucid and helpful reports from one of them. Not only did he spot gaps and mistakes in some of my initial proofs, but his insights eventually led to a much more direct construction of the refined height pairing, making it more integral and completely avoiding my initial recourse to semi-stable models (see Remark \ref{r2.2} c)). 

\subsection*{Notation and conventions} We try and follow Fulton's notation in \cite{fulton} as much as possible. In particular, given a morphism of $k$-schemes $f:X\to Y$, we write $\gamma_f$ for the associated graph morphism $X\to X\times_k Y$ and $\delta_X$ for $\gamma_{1_X}$; if $f$ admits refined Gysin morphisms as in loc. cit., Ch. 6 and 8, we write them $f^!$ and sometimes use the notation $f^*$ for ordinary Gysin morphisms.

We usually abbreviate the notation $\times_k$ (fibre product over $k$) to $\times$, and re-establish it when it may be confused with other fibre products.

We shall encounter $k$-schemes essentially of finite type, being of finite type over some localisation of $B$. We shall sometimes commit the abuse of treating them as if they were of finite type: for example, call them smooth even if they really are essentially smooth, and take (refined) Gysin morphisms associated to morphisms between them even if these morphisms are not of finite type. This is easily justified by the fact that Chow groups commute with inverse limits of open immersions  \cite[Lemma IA.1]{bloch2}.

\numberwithin{equation}{section}

\section{An elementary reduction}\label{s1.1} 

\subsection{Intersection on regular $K$-schemes}\label{r1.1} Let $K$ be a field. If $\car K=0$, every regular $K$-scheme $X$, separated of finite type, is smooth, so the intersection theory of \cite[Ch. 8]{fulton} applies. Here we point out that this is also true in characteristic $p>0$: it will be needed in \S\S \ref{s4.2} ff.

We may assume $K$ to be finitely generated over its (perfect) subfield $k=\F_p$, and  $X$ (regular) to be irreducible of dimension $d$.  We may find a smooth connected separated $k$-scheme $B$ of finite type with generic point $\eta=\Spec K$, and a dominant morphism $f:\sX\to B$ with $\sX$ $k$-smooth, of generic fibre $X$. We have the intersection pairing of \cite[\S 8.1]{fulton}: for $i,r\ge 0$,
\begin{equation}\label{eq1}
CH^i(\sX)\times CH^{d+r-i}(\sX)\by{\cdot} CH^{d+r}(\sX)
\end{equation}
which commutes with base change by \cite[Prop. 6.6 (c) and 8.3 (a)]{fulton}. Then \eqref{eq1} induces an intersection product on $X$ by passing to the limit. If $f_1:\sX_1\to B_1$ is another choice, then $B$ and $B_1$ share a common open subset with isomorphic fibres, so this intersection product is  independent of the choice of $(B,f)$.

Suppose moreover $X$ and $f$ proper. Composing \eqref{eq1} with $f_*$, we get a pairing
\begin{equation}\label{eq2}
CH^i(\sX)\times CH^{d+r-i}(\sX)\by{\langle, \rangle} CH^{r}(B)
\end{equation}
 
For the same reason, numerical equivalence makes sense on $X$ via \eqref{eq2}, and does not depend on any choice.

\subsection{The set-up}\label{s1.2} Let now $k$ be any perfect field; we place ourselves in the situation $(B,\sX,f)$ of \S\ref{r1.1} with $f$ proper, and let $f':X\to \eta$ be the generic fibre of $f$. In particular, the observations of \S\ref{r1.1} apply to $X$.

For a subscheme  $Z$ of $B$, write $\sX_Z=f^{-1}(Z)$, $\iota:\sX_Z\inj \sX$ for the corresponding immersion and $f_Z:\sX_Z\to Z$ for the projection induced by $f$. We extend these notations to pull-backs by a morphism $Z\to B$ when there is no ambiguity in the context.

\begin{lemma}\label{l1} Suppose that $\codim_B Z>r$. Then \eqref{eq2} factors through a pairing
\[CH^i(\sX-\sX_Z)\times CH^{d+r-i}(\sX-\sX_Z)\by{\langle, \rangle} CH^{r}(B).\]
\end{lemma}

\begin{proof} We have
\begin{equation}\label{eq00}
CH^r(B)\iso CH^r(B-Z).
\end{equation}
\end{proof}

We shall use the case $r=1$ of this lemma in the rest of this paper. 

\begin{rques}\label{r1} a) Let $\sZ$ be the locus of non-smoothness of $f$. If $f'$ is smooth, $f(\sZ)$ is a proper closed subset of $B$, hence contains only finitely many points of $B^{(1)}$, the set of codimension $1$ points of $B$.

b) If $\delta=1$, any proper surjective morphism $\phi$ from an irreducible $k$-variety $V$ to $B$ is flat \cite[Ch. II, Prop. 9.7]{hartshorne}; in general, this is true after base-changing to the local scheme at any point $b\in B^{(1)}$. If $F\subset V$ is the (closed) locus of non-flatness of $\phi$, the closed subset $\phi(F)$ is therefore of codimension $\ge 2$ in $B$. This shows that one may reduce to $\phi$ flat by removing a closed subset of codimension $\ge 2$ from $B$. This technique may be applied to $f$ if necessary; a variant will be used in the proof of Proposition \ref{p2.2}.
\end{rques}

Let $CH^i_\num(X)$ denote the subgroup of $CH^i(X)$ formed of cycles numerically equivalent to $0$; write $j$ for the inclusion $X\inj \sX$.

\begin{lemma}\label{l4}
For $\alpha\in CH^i(\sX)$, the following are equivalent:
\begin{enumerate}
\item $j^*\alpha\in CH^i_\num(X)$;
\item for any $\beta\in CH^{d-i}(\sX)$, $f_*(\alpha\cdot \beta)=0$.
\end{enumerate}
\end{lemma}

\begin{proof} (2) $\Rightarrow$ (1) because of the surjectivity of $j^*$ and the formula
\begin{equation}\label{eq5}
\jmath^*f_*(\alpha\cdot \beta) =f'_*j^*(\alpha\cdot \beta) = f'_*(j^*\alpha\cdot j^*\beta)
\end{equation}
\cite[Prop. 1.7 and 8.3 (a)]{fulton} where $\jmath:\eta\inj B$ is the inclusion, and (1) $\Rightarrow$ (2) because of \eqref{eq5} and the injectivity of $\jmath^*:CH^0(B)\to CH^0(\eta)$.
\end{proof}

\section{The refined height pairing}\label{s2.0}

We keep the set-up of \S \ref{s1.2}. 

\subsection{Review of Fulton's refined Gysin morphisms} Let $f:X\to Y$ be a morphism of algebraic $k$-schemes, of constant dimensions $d_X$ and $d_Y$ for simplicity, and let $d=d_Y-d_X$. In certain cases, Fulton associates to $f$ ``refined Gysin morphisms''
\[f^!:CH_*(Y')\to CH_{*-d}(X\times_Y Y')\]
for any $Y$-scheme $Y'$; these morphisms are compatible with push-forward, pull-back and intersection products in the sense of \cite[Def. 17.1]{fulton}. Such collections of morphisms are called \emph{orientations} in loc. cit., \S 17.4. Orientable morphisms are
\begin{itemize}
\item flat morphisms (loc. cit., Th. 1.7),
\item regular embeddings (loc. cit., \S\S 6.2, 6.4), 
\item more generally, l.c.i. morphisms (loc. cit., \S 6.5),
\item morphisms to a smooth $Y$ (loc. cit., Def. 8.1.2). 
\end{itemize}

The definitions of $f^!$ agree when $f$ is of several of these forms at the same time, e.g. loc. cit., Prop. 8.1.2. The assignment $f\mapsto f^!$ is functorial in certain cases, many of which are summarised in loc. cit., Ex. 17.4.6.

Since it is difficult to find a unified statement of all these compatibilities in \cite{fulton}, we shall strive to give precise references for all those we use; the above reminder should only be viewed as a guide to the reader.

We shall very often use the following situation, that we record as a lemma.

\begin{lemma}\label{l2.5} Let
\[\begin{CD}
S'@>f'>> T'\\
@Vg'VV @VgVV\\
S@>f>> T
\end{CD}\]
be a Cartesian square of $k$-schemes, where $g$ is proper and $f$ is an l.c.i. morphism. Then\\
a) One has
\[ f^!g_* = g'_*f^! \]
as homomorphisms from $CH_*(T')$ to $CH_*(S)$.\\
b) If $f'$ is also an l.c.i. morphism, of same codimension, then $f^!={f'}^!$.\\
c) If $f$ and $g$ are two composable l.c.i. morphisms, then $(g\circ f)^!=f^!\circ g^!$.
\end{lemma}

\begin{proof} This follows from \cite[Th. 6.6 (c)]{fulton}.
\end{proof}

\subsection{Admissible cycles} \label{s2.1} Let $b\in B^{(1)}$; write $Z=\overline{\{b\}}$. 
  Recall the cap-product \cite[p. 131]{fulton}
\begin{align*}
\cdot_\iota:CH^i(\sX)\times CH_l(\sX_Z)&\to CH_{l-i}(\sX_Z) \\
(\alpha,\beta) &\mapsto \gamma_\iota^!(\beta\times \alpha)
\end{align*}
where $\iota$ is the closed immersion $\sX_Z \inj \sX$. 

Take $l=\delta+i-1$. Composing with $(f_Z)_*$, we get a pairing
\begin{align}
\langle, \rangle_b:CH^i(\sX)\times CH_{\delta+i-1}(\sX_Z)&\to CH_{\delta-1}(Z)=CH^0(Z)=\Z\label{eq2.1}\\
\langle\alpha,\beta \rangle_b&=(f_Z)_*(\alpha\cdot_\iota \beta).\notag
\end{align}

We record two useful formulas:
\begin{equation}\label{eq3}
\alpha \cdot \iota_*\beta=\iota_*(\alpha \cdot_\iota \beta) \in CH_{l-i}(\sX)
\end{equation}
which follows from  Lemma \ref{l2.5} applied to the Cartesian diagram
\[\begin{CD}
\sX_Z @>\gamma_\iota >> \sX_Z\times \sX\\
@V \iota VV @V \iota\times 1 VV\\
\sX @>\Delta_\sX>> \sX\times \sX
\end{CD}\]
of regular embeddings of codimension $d+\delta$. Hence
\begin{equation}\label{eq6}
f_*(\alpha\cdot \iota_*\beta)=f_*\iota_*(\alpha\cdot_\iota\beta)=\iota'_*\langle\alpha,\beta\rangle_b,
\end{equation} 
where $\iota'$ is the closed immersion $Z\inj B$.

\begin{defn}\label{r2.1} With the above notation, we set
\begin{multline*}
CH^i(\sX)_b^0=\{\alpha\in CH^i(\sX)\mid j^*\alpha\in CH^i_\num(X)\\
 \text{ and } \langle\alpha,\beta\rangle_b=0\ \forall \beta\in CH_{\delta+i-1}(\sX_Z)\}
 \end{multline*}
for $b\in B^{(1)}$, and
\[CH^i(\sX)^0=\bigcap_{b\in B^{(1)}}CH^i(\sX)_b^0.\]
We call the cycles in $CH^i(\sX)^0$ \emph{admissible}.
\end{defn}

Even if it is not apparent anymore, this definition was inspired by \cite[Assumption 2]{bloch} and \cite[1.2]{beilinson}.

\begin{rques}\label{r2.2} 
a) One should be careful that $CH^i(\sX)^0$ does not contain $\Ker j^*$ in general. For example, let $B=\A^1=\Spec k[t]$ and let $\sX$ be the hypersurface in $B\times \P^2$ with (partly) homogeneous equation $tX_0^2=X_1X_2$. 
Then the pull-back of the curve $(t=X_1=0)$, viewed as a codimension $1$ cycle on $\sX$, to the curve $(t=X_2=0)$, is  the point $(0,(1:0:0))$ which is not numerically equivalent to $0$. On the other hand, if $f$ is smooth above $\Spec \sO_{B,b}$ for a $b\in B^{(1)}$, then any element of $\Ker j^*$ vanishes when restricted to $\sX_b$ thanks to \cite[\S 20.3]{fulton}. So this caveat only involves finitely many exceptional $b$'s.

b) The pairing \eqref{eq2.1} makes sense for any $b\in B$ (replacing $CH_{\delta+i-1}(Z)$ by $CH_{\delta+i-r}(Z)$ if $b\in B^{(r)}$), and defines an equivalence relation $\alpha\equiv_b 0$ if $\langle \alpha,\beta\rangle_b=0$ for any $\beta\in CH_{\delta+i-r}(Z)$. One can show that $\alpha\equiv_{b'}0\Rightarrow \alpha\equiv_b0$ if $b'$ is a specialisation of $b$; in particular,  the condition $j^*\alpha\in CH^i_\num(X)$ is superfluous in the definition of $CH^i(\sX)_b^0$, thanks to Lemma \ref{l4}. We shall not use these facts in the present paper, so the rather long proof is omitted (see \cite{pointwise}). 

c) Let $b\in B^{(1)}$. Suppose that all the irreducible components $\sX_b^\lambda$ of $\sX_b$ are of dimension $d$ and smooth over $k(b)$. Then it is easy to see that $\alpha\equiv_b 0$ if and only if $ \kappa_\lambda^!\alpha\in CH^i_\num(\sX_b^\lambda)$
for all $\lambda$, where $\kappa_\lambda:\sX_b^\lambda \inj \sX$ is the inclusion. Our initial approach to the refined height pairing was based on such models; they are not necessary anymore.
\end{rques}

We obviously have

\begin{lemma}\label{l3}  The quotient $CH^i(\sX)/CH^i(\sX)^0$ is torsion-free.\qed
\end{lemma}

\subsection{Comparison with numerical and homological equivalence}\label{s2.3}

\begin{prop}\label{p2.4} If $B$ is projective (hence $\sX$ is $k$-proper), we have $CH^i_\num(\sX)\allowbreak\subseteq CH^i(\sX)^0$.
\end{prop}

\begin{proof} Let $\alpha\in CH^i_\num(\sX)$: we want to show that $\alpha\in CH^i(\sX)^0$. Let first $\beta\in CH^{d-i}(\sX)=CH_{\delta+i}(\sX)$. Choose a $0$-cycle $z\in CH_0(B)$ of nonzero degree. Then
\[0=\deg(\alpha\cdot \beta\cdot f^*z) = \deg(f_*(\alpha\cdot \beta)\cdot z)=f_*(\alpha\cdot \beta)\deg(z)\]
hence $f_*(\alpha\cdot \beta)=0$, and we conclude that $j^*\alpha\in CH^i_\num(X)$ by Lemma \ref{l4}. 

Let now $b\in B^{(1)}$, and $Z=\overline{\{b\}}$ as above. Let $\beta\in CH_{\delta+i-1}(\sX_Z)$. We have this time
\[0=f_*(\alpha\cdot \iota_*\beta\cdot f^*z) =  f_*(\alpha\cdot \iota_*\beta)\cdot z\]
for any $z\in CH_1(B)=CH^{\delta-1}(B)$, i.e. $f_*(\alpha\cdot \iota_*\beta)=\iota'_*\langle\alpha,\beta\rangle_b\in CH^1_\num(B)$ (see \eqref{eq6}). But $\iota'_*:\Z=CH^0(Z)\to CH^1(B)/CH^1_\num(B)$ is injective since $Z$, as an irreducible divisor on a smooth projective variety, is not numerically equivalent to $0$ (compare \cite[Ch. I, Th. 1.21]{debarre}). Therefore $\langle\alpha,\beta\rangle_b=0$, as requested.
\end{proof}

Let now $l$ be a prime number invertible in $k$. We have a composition
\begin{equation}\label{eq2.6}
CH^i(\sX)\to H^{2i}(\sX_{\bar k},\Q_l(i)) \to H^0(B_{\bar k},R^{2i}f_*\Q_l(i))
\end{equation}
where the first map is the (geometric) cycle class map. Write $CH^i_l(\sX)$ (resp. $CH^i(\sX)^0_{\cB,l}$ for the kernel of the first map (resp. of their composition): the latter group is introduced by analogy to \cite[1.2]{beilinson} which is the special case $\delta=1$, $k$ algebraically closed. We obviously have $CH^i_l(\sX)\subseteq CH^i(\sX)^0_{\cB,l}$.
\bigskip

 The following is parallel to Proposition \ref{p2.4}, without assuming $B$ projective. It will be used in Remarks \ref{r2.5} and \ref{r6.1} a) and in Proposition \ref{r2.3}.

\begin{prop}\label{p2.1} At least in characteristic $0$, $CH^i(\sX)^0_{\cB,l}\allowbreak\subseteq CH^i(\sX)^0$.
\end{prop}

\begin{proof} Let $\alpha\in CH^i(\sX)^0_{\cB,l}$. Then $\alpha$ vanishes in $H^0(K\bar k,R^{2i}f_*\Q_l(i))=H^{2i}(X\otimes_k \bar k,\Q_l(i))$, hence a fortiori in $H^{2i}(X\otimes_K \bar K,\Q_l(i))$: this means that $j^*\alpha$ is $l$-adically homologically equivalent to $0$, hence also numerically equivalent to $0$. This part of the proof works in all characteristics.

We now give the sequel of the proof in characteristic $0$: to oversimplify, it follows by functoriality from the fact that the cycle class map is injective in codimension $0$ (sic). (So this argument is geometrically cheaper than the one for Proposition \ref{p2.4}.) 

We may assume $k$ finitely generated and choose an embedding of $k$ in $\C$. By Artin's comparison theorem, $CH^i(\sX)^0_{\cB,l}=\Ker(CH^i(\sX)\to H^0_B(B_\C,R^{2i}f_*\Q(i))$ where $H_B$ denotes Betti (or analytic) cohomology. Let $b\in B^{(1)}$, and let $Z$, $\iota$, $\beta$ be as in Definition \ref{r2.1}. To show that $\langle \alpha,\beta\rangle_b=0$ in $CH^0(Z)\iso CH^0(Z_\C)$, we may assume $k=\C$ and drop all Tate twists. 

In \cite[Ch. 19]{fulton}, a cycle class map $\cl$ is defined for Chow groups of complex, possibly singular, varieties, with values in their Borel-Moore homology and we have the formula
\begin{equation}\label{eq2.3a}
\cl(\alpha\cdot_\iota \beta) = {\iota'}^*(\cl(\alpha))\cap \cl(\beta)\in H_{2\delta -2}(\sX_Z)
\end{equation}
 \cite[Prop. 19.2]{fulton}, where $\iota'$ is the closed immersion $Z\inj B$ as in the previous proof, hence
 \begin{equation}\label{eq2.3}
 \cl(\langle\alpha,\beta\rangle_b) =(f_Z)_*(\iota^*(\cl(\alpha))\cap \cl(\beta))\in H_{2\delta -2}(Z)
 \end{equation}
 since $\cl$ commutes with push-forwards, by definition and \cite[Lemma 19.1.2]{fulton}. 
 
 It now suffices to show that the right hand side of  \eqref{eq2.3} vanishes since $CH_{\delta -1}(Z)\to H_{2\delta -2}(Z)$ is injective, as one sees by reducing to $Z$ smooth by removing from it a proper closed subset. For this, it suffices to show that the pairing
\begin{equation}\label{eq2.4}
H^{2i}(\sX)\times H_{2\delta -2 + 2i}(\sX_Z)
\to H_{2\delta-2}(Z),
\end{equation}
given by $(x,y)\mapsto (f_Z)_*(\iota^*x\cap y)$, factors through $H^{0}(B,R^{2i}f_*\Q)\times H_{2\delta -2 + 2i}(\sX_Z)$.

We switch by  Poincaré duality from the Borel-Moore homology of $\sX_Z$ (resp. $Z$) to the cohomology of the smooth variety $\sX$ (resp. $B$) with supports in $\sX_Z$ (resp. in $Z$). Then \eqref{eq2.4} becomes the composition
\begin{equation}\label{eq2.5}
H^{2i}(\sX)\times H^{2d +2 - 2i}_{\sX_Z}(\sX)\by{\cap}H^{2d +2}_{\sX_Z}(\sX) \by{f_*} H^{2}_Z(B),
\end{equation}
where $\cap$ is the usual cap-product. The (global) trace map $f_*$ factors as a composition
\[H^{2d+2}_{\sX_Z}(\sX)\to H^0_Z(B,R^{2d+2}f_*\Q)\by{(\Tr_f)_*} H^2_Z(B)
\]
where $\Tr_{f}$ is the local trace map in étale cohomology for the proper morphism $f$. Thus, \eqref{eq2.5} factor through the map
\[
H^{2i}(\sX)\times H^{2d +2 - 2i}_{\sX_Z}(\sX)\to H^0(B,R^{2i}f_*\Q)\times H^0_Z(B,R^{2d+2-i}f_*\Q)
\]
as requested.

In positive characteristic, the leap of faith is that \eqref{eq2.3a} and \eqref{eq2.3} hold for the cycle class maps defined in $l$-adic Borel-Moore homology \cite[\S 6]{laumon}. The commutation with push-forwards causes no problem, and  \eqref{eq2.3a}  indeed appears in \cite[Th. (7.2)]{laumon}, except that the extraordinary cap-product $\cdot_\iota$ (defined in \cite[2.1.1]{verdier} using intersection multiplicities) should be shown to agree with Fulton's. (This is suggested in the Notes and references of \cite[Ch. 19]{fulton}, see also loc. cit., p. 382.)\footnote{Using Olsson's theorem that the cycle class maps commute with refined Gysin homomorphisms \cite[Th. 2.34]{olsson}, it would suffice to show the identity $\gamma_\iota^!(x\times y) =  \iota^* y\cap x$ in Borel-Moore $l$-adic homology.}
 
This being accepted, the same argument goes through.
\end{proof}

\begin{rque}\label{r2.4} 
As a referee pointed out, there is an important conceptual difference between 
$CH^i(\sX)^0_{\cB,l}$ and $CH^i(\sX)^0$: by the smooth and proper base change, we have the equality $\Ker(H^{2i}(\sX_{\bar k},\Q_l(i)) \to H^{2i}(X_{\bar k},\Q_l(i)))=\Ker(H^{2i}(\sX_{\bar k},\Q_l(i)) \to H^0(U_{\bar k},R^{2i}f_*\Q_l(i)))$ for any open subset $U\subseteq B$ over which $f$ is smooth. Thus, the condition $\alpha\in CH^i(\sX)^0_{\cB,l}$ for $\alpha\in CH^i(\sX)$ only has to be checked at the generic fibre and at the ``bad fibres'' of $f$. This contrasts with the case of $CH^i(\sX)^0$, see Remark \ref{r2.2} a). See also Remarks \ref{r6.1} further down.
\end{rque}

\subsection{Global height pairing} The following proposition is the key point of this paper.

\begin{prop}\label{p1} Let $\alpha\in CH^i(\sX)^0$.  If $\beta\in CH^{d+1-i}(\sX)$ and  $j^*\beta=0$, then $f_*(\alpha \cdot \beta)=0$ in $CH^1(B)$.
\end{prop}

\begin{proof} By \cite[Prop. 1.8]{fulton}, write $\beta = \iota_* \beta'$ with $\beta'\in CH_{\delta+i-1}(\sX_Z)$ for some proper closed subset $Z\subset B$, where $\iota:\sX_Z\inj \sX$ is the inclusion.  We may assume that $\beta'$ is the class of an irreducible cycle, hence take $Z$ irreducible. If $\codim_B Z>1$, the result follows from Lemma \ref{l1}. 
If $Z=\overline{\{b\}}$ for $b\in B^{(1)}$, the conclusion follows from \eqref{eq6}.
\end{proof}

The proof of the following lemma is in the same spirit, so we include it here. It will be used in the proof of Proposition \ref{p5} (ii).

\begin{lemma} \label{l2.4} Let $b_1,\dots,b_n$ be a finite set of points on $B^{(1)}$ and let $Z=\overline{\{b_1,\dots, b_n\}}$. Then one has $(f_Z)_*(\alpha\cdot_\iota\beta)=0$ for any $\alpha\in CH^i(\sX)^0$ and any $\beta\in CH_{\delta+i-1}(\sX_Z)$, where $\iota$ is the closed immersion $\sX_Z\inj \sX$.
\end{lemma}

\begin{proof} We may asume that $\beta$ is the class of an irreducible cycle $\beta'$; then $\beta'$ is supported on $\sX_{Z_r}$ for some $r$, where $Z_r = \overline{\{b_r\}}$. Let $\kappa:\sX_{Z_r}\inj \sX_Z$ be the corresponding closed immersion, and let $\iota_r=\iota\kappa$:  by applying again Lemma \ref{l2.5} to the obvious Cartesian square involving $\kappa$, we get the identity
\[\alpha\cdot_\iota \kappa_*\beta'=\kappa_*(\alpha\cdot_{\iota_r} \beta')\]
etc.
 \end{proof}

\begin{defn}\label{d2.1} Let $CH^i(X)^0_f$ be the image of $CH^i(\sX)^0$ in $CH^i(X)$. By Proposition \ref{p1}, \eqref{eq2} induces a pairing
\[CH^i(\sX)^0\times CH^{d+1-i}(X)\to CH^{1}(B)\]
hence, swapping $i$ with $d+1-i$, a ``height'' pairing
\begin{equation}\label{eq4}
\langle,\rangle_f:CH^i(X)^0_f\times CH^{d+1-i}(X)^0_f\to CH^{1}(B).
\end{equation}
\end{defn}

We shall see in the next section (Propositions \ref{l6} and \ref{r4}) that neither $CH^i(X)^0_f$ nor $\langle,\rangle_f$ depends in the choice of $f$.

\subsection{Comparison with the pairing of Rössler-Szamuely}

\begin{prop}\label{r2.5} The pairing \eqref{eq4} is compatible with the pairing \eqref{eq01} of the introduction on the subgroups $CH^i(\sX)^0_{\cB,l}$ and $CH^{d+1-i}(\sX)^0_{\cB,l}$ of Proposition \ref{p2.1}.
\end{prop}

\begin{proof}
Using cup-product and push-forward in $l$-adic cohomology:
\begin{multline}\label{eq2.7}
 H^{2i}(\sX_{\bar k},\Q_l(i))\otimes H^{2(d+1-i)}(\sX_{\bar k},\Q_l(d+1-i))\\\by{\cup} H^{2(d+1)}(\sX_{\bar k},\Q_l(d+1))
 \by{f_*} H^2(B_{\bar k},\Q_l(1))
 \end{multline}
we get from \eqref{eq2.6} a pairing
\begin{equation}\label{eq2.8}
CH^i(\sX)\otimes CH^{d+1-i}(\sX)\to H^2(B_{\bar k},\Q_l(1))
\end{equation}
which is evidently compatible with \eqref{eq2} (for $r=1$). On the other hand, the Leray spectral sequence
\begin{equation}\label{eq2.13}
H^r(B_{\bar k},R^sf_*\Q_l(i))\Rightarrow H^{r+s}(\sX_{\bar k},\Q_l(i))
\end{equation}
 yields Abel-Jacobi maps
\begin{equation}\label{eq2.11}
\AJ^i_B:CH^i(\sX)^0_{\cB,l}\to   H^1(B_{\bar k},R^{2i-1}f_*\Q_l(i)).
\end{equation}

We have a pairing parallel to \eqref{eq2.7}:
\begin{multline}\label{eq2.12}
H^1(B_{\bar k},R^{2i-1}f_*\Q_l(i))\otimes H^1(B_{\bar k},R^{2(d-i+1)-1}f_*\Q_l(d-i+1))\\\by{\cup} H^2(B_{\bar k},R^{2d}f_*\Q_l(i)))
 \by{\Tr_f} H^2(B_{\bar k},\Q_l(1))
 \end{multline}
which is compatible with the former via \eqref{eq2.13}. This implies that the restriction of \eqref{eq2.8} to $CH^i(\sX)^0_{\cB,l}\otimes CH^{d+1-i}(\sX)^0_{\cB,l}$ is \emph{compatible with} \eqref{eq4} via Proposition \ref{p2.1}, i.e. that the diagram
\begin{equation}\label{eq2.9}
\begin{CD}
CH^i(\sX)^0_{\cB,l}\otimes CH^{d+1-i}(\sX)^0_{\cB,l}@>>> H^2(B_{\bar k},\Q_l(1))\\
@VVV @AAA\\
CH^i(\sX)^0\otimes CH^{d+1-i}(\sX)^0@>>> CH^1(B)
\end{CD}
\end{equation} 
commutes. 

On the other hand, the height pairing of \cite{RS} is defined on $CH^i_l(X)\otimes CH^{d+1-i}_l(X)$, also with values in $H^2(B_{\bar k},\Q_l(1))$. More precisely, by loc. cit., Th. Prop. 2.3, if $\alpha\in CH^i_l(X)$, $j:U\inj B$ is an open subset over which $f$ is smooth and $\alpha_U$ is a lift of $\alpha$ to $CH^i(\sX_U)$, then $\AJ^i_U(\alpha_U)\in H^1(U_{\bar k},R^{2i-1}(f_U)_*\Q_l(i))$ lies in the subgroup $H^{1-\delta}(B_{\bar k},j_{!*}R^{2i-1}(f_U)_*\Q_l(i))$ (loc. cit., Prop. 2.1), and the height pairing of Rössler and Szamely is defined by \eqref{eq2.12} on these subgroups. Let $\sF=R^{2i-1}(f_U)_*\Q_l(i)\break =j^*R^{2i-1}f_*\Q_l(i)$. Since $j^*j_{!*} \sF=\sF$ \cite[Rem. 1.4.14.1]{bbd}, the image of $H^1(B_{\bar k},R^{2i-1}f_*\Q_l(i))$ in $H^1(U_{\bar k},R^{2i-1}(f_U)_*\Q_l(i))$ is contained in \break $H^{1-\delta}(B_{\bar k},j_{!*}R^{2i-1}(f_U)_*\Q_l(i))$.
\end{proof}

\section{Independence from the (smooth) model}

\subsection{Review of the Corti-Hanamura category}\label{s2.2}

A morphism $f:\sX\to B$ as in \S \ref{s1.1} defines an object in the Corti-Hanamura category $\CHC(B)$ of  \cite[Def. 2.8]{ch}\footnote{Except that $f$ is assumed projective in \cite{ch}; proper is sufficient to apply its formalism.}. Given two such objects $f_i:\sX_i\to B$ ($i=1,2$), morphisms in $\CHC(B)$ are defined by relative correspondences:
\[\CHC(B)(\sX_1,\sX_2)= CH_{\dim \sX_2}(\sX_1\times_B \sX_2)=CH^{\dim X_1}(\sX_1\times_B \sX_2)\]
where $X_1$ is the generic fibre of $\sX_1$.

If $f_3:\sX_3\to B$ is a third object, the composition of two such correspondences $u:\sX_1\to \sX_2$ and $v:\sX_2\to \sX_3$ is defined as
\begin{equation}\label{eq03}
v\bullet u = (p^{1,2,3}_{1,3})_*\delta_2^!(u\times_k v) 
\end{equation} 
where $\delta^!$ is the refined Gysin morphism from \cite[\S 6.2]{fulton} associated to the (regular immersion) diagonal $\delta_2:\sX_2\to \sX_2\times_k \sX_2$ in the (augmented) Cartesian square
\begin{equation}\label{eq02}\Small
\begin{CD}
(\sX_1\times_B \sX_3) @<p^{1,2,3}_{1,3}<< \sX_1\times_B\sX_2\times_B \sX_3 @>\Delta>>  (\sX_1\times_B \sX_2)\times_k (\sX_2\times_B \sX_3)\\
&&@VVp^{1,2,3}_{2} V @VV p^{1,2}_{2} \times p^{2,3}_{1}V\\
&& \sX_2@>\delta_2>> \sX_2\times_k \sX_2 
\end{CD}
\end{equation}
and the notation for the projections is self-evident.

As usual, one can generalise this to ``graded correspondences''
\[\CHC(B)_r(\sX_1,\sX_2)= CH_{\dim \sX_2-r}(\sX_1\times_B \sX_2)= CH^{\dim X_1+r}(\sX_1\times_B \sX_2)\] 
and reduce these graded correspondences to ordinary ones if one wishes, by using the projective bundle formula \cite[Th. 3.3 (b)]{fulton}.

Since $\Delta$ is also a regular immersion of the same codimension as $\delta$ (namely, $\dim \sX_2$), we may apply Lemma \ref{l2.5} b) which gives 
\begin{equation}\label{eq2.10}
\delta_2^!(v\times_k u)=\Delta^!(v\times_k u).
\end{equation} 

If the $f_i$ are smooth, we also have a ``classical'' composition of correspondences à la Deninger-Murre \cite{den-murre}:
\[v\circ u = (p^{1,2,3}_{1,3})_*\big((p^{1,2,3}_{2,3})^*v\cdot (p^{1,2,3}_{1,2})^*u\big).\]

\begin{lemma}\label{l2.1} a) In the above case, $v\circ u = v\bullet u$.\\
b) The category $\CHC(B)$ is contravariant for smooth $k$-morphisms $\phi:C\to B$.\\
c) The pro-open immersion $j$ defines a functor to the category of Chow correspondences over $K$ from the full subcategory of $\CHC(B)$ consisting of those $f:\sX\to B$ whose generic fibre is smooth. 
\end{lemma}

\begin{proof} a) We use \eqref{eq2.10}. We have the Cartesian square
\[\begin{CD}
\sX_1\times_B\sX_2\times_B \sX_3 @>\Delta_1>> (\sX_1\times_B\sX_2\times_B \sX_3)\times_k  (\sX_1\times_B\sX_2\times_B \sX_3) \\
@V||VV @VV p^{1,2,3}_{2,3}\times p^{1,2,3}_{1,2} V\\
\sX_1\times_B\sX_2\times_B \sX_3 @>\Delta>> (\sX_2\times_B \sX_3)\times_k (\sX_1\times_B \sX_2)
\end{CD}\]
in which all morphisms are l.c.i. morphisms, hence
\[\Delta^!(v\times_k u)=\Delta_1^!(p^{1,2,3}_{2,3}\times p^{1,2,3}_{1,2})^!(v\times_k u)\]
by Lemma \ref{l2.5} c), 
\[(p^{1,2,3}_{2,3}\times p^{1,2,3}_{1,2})^!(v\times_k u)=(p^{1,2,3}_{2,3}\times p^{1,2,3}_{1,2})^*(v\times_k u) = (p^{1,2,3}_{2,3})^*v \times_k (p^{1,2,3}_{1,2})^*u\]
by \cite[Prop. 6.6 (b)]{fulton}, and finally
\[\Delta_1^!\big((p^{1,2,3}_{2,3})^*v \times_k (p^{1,2,3}_{1,2})^*u\big)= (p^{1,2,3}_{2,3})^*v \cdot (p^{1,2,3}_{1,2})^*u\]
by definition of the intersection product on smooth varieties \cite[p. 131]{fulton}.

In b), the statement means that $\phi$ defines a functor $\phi^*:\CHC(B)\allowbreak\to \CHC(C)$, given by fibre product. It is defined on objects by the smoothness of $\phi$, and on morphisms because smooth morphisms are flat. To check that it respects composition involves chasing in the Cartesian cube obtained by pulling back the square of $B$-schemes \eqref{eq02} along the morphism $C\times_B C\to B$, and then further pulling back along the diagonal $\delta':C\to C\times_B C$; this latter operation is unnecessary if $C$ is an open subset of $B$. The first step involves \cite[Prop. 6.6]{fulton} as in the proof of a), to take care of the flat l.c.i morphisms $C\times_B(\sX_i\times_B \sX_j)\to \sX_i\times_B \sX_j$; the second step uses the fact that $\delta'$ is a regular immersion.

c) follows from a), b) and  \cite[Lemma IA.1]{bloch2}, since $U\times_B\sX$ is smooth over $U$ for a suitable open subset $U$ of $B$ for $\sX$ as in the statement.
\end{proof}

\begin{rque} The associativity of the composition $\bullet$ is not proven in \cite{ch}. It will not be used here and is left to the reader. See nevertheless Remark \ref{r11}.
\end{rque}

As a special case of \eqref{eq03}, take $\sX_3=B$: we get pairings
\begin{gather*}
CH^{\dim X_2+r}(\sX_1\times_B \sX_2)\otimes CH^i(\sX_2)\to CH^{i+r}(\sX_1)\\
(\psi,\alpha)\mapsto \psi^*\alpha:= (p_{1})_*\delta_2^!( \psi\times_k\alpha)
\end{gather*}
compatible via $j^*$ with the usual action of correspondences over $K$, by Lemma \ref{l2.1} c). For clarity, we repeat \eqref{eq03} in this special case:
\begin{equation}\label{eq18}
\begin{CD}
\sX_1@<p_1<< \sX_1\times_B \sX_2@>\gamma_{p_2}>> (\sX_1\times_B \sX_2)\times_k \sX_2\\
&&@Vp_2 VV @V p_2\times 1VV\\
&&\sX_2@>\delta_2>> \sX_2\times_k \sX_2
\end{CD}
\end{equation}
where $\gamma_{p_2}$ is the graph of $p_2:=p_2^{1,2}$.

We also write $\psi_*$ for $(^t\psi)^*$.

As an even more special case, when $\sX_1=B$: writing $\beta$ rather than $\psi$, we recover the pairing \eqref{eq2}
\begin{equation}\label{eq34}
\langle \alpha,\beta\rangle=(f_2)_*(\alpha\cdot \beta)=(f_2)_*\delta_2^!(\alpha\times_k \beta)=\beta^*\alpha\in CH^*(B).
\end{equation}

\begin{lemma}\label{l3.3} Let $(\alpha,\beta)\in CH^i(\sX_1)\times CH^{d_1-i+1}(\sX_2)$ and $\psi\in CH^{d_2}(\sX_1\times_B \sX_2)$. Then
\[\langle \psi^*\alpha, \beta\rangle=\langle \alpha,\psi_*\beta\rangle.\]
\end{lemma}

\begin{proof} For clarity, write $\delta_i$ for the diagonal map $\sX_i\to \sX_i\times_k \sX_i$. As in the proof of Proposition \ref{l6}, let $p_i$ be the projection $\sX_1\times_B \sX_2\to \sX_i$. Developing, the identity to be proven is
\begin{equation}\label{eq19bis}
(f_1)_*((p_1)_*\delta_2^!(\psi \times  \alpha)\cdot  \beta)  = (f_2)_*( \alpha\cdot (p_2)_*\delta_1^!(^t\psi \times  \beta)).
\end{equation}

Let $\lambda =\delta_2^!(\psi \times  \alpha)$. We have
\[(p_1)_*\lambda \cdot  \beta = \delta_1^!((p_1)_*\lambda \times  \beta) =  \delta_1^!(p_1\times 1)_*(\lambda\times \beta)= (p_1)_* \delta_1^!(\lambda\times \beta)\]
by Lemma \ref{l2.5} a). Similarly, if $\lambda'=\delta_1^!(^t\psi \times  \beta)$ and $\lambda'':=\delta_1^!(\psi \times  \beta)$:
\[  \alpha\cdot (p_2)_*\lambda' = (p_2)_* \delta_2^!(\alpha\times \lambda')=(p_2)_* \delta_2^!( \lambda''\times\alpha).\]

Since $f_1p_1=f_2p_2$, to show \eqref{eq19bis} it suffices to show that
\[ \delta_1^!(\lambda\times \beta)= \delta_2^!(\lambda''\times\alpha).\]

We now observe that since $\sX_2$ is smooth, $\gamma_{p_2}$ is also a regular embedding in \eqref{eq18}, hence $\delta_2^!=\gamma_{p_2}^*$ (non-refined Gysin map) by Lemma \ref{l2.5} b) (see also \eqref{eq2.10}); similarly, $\delta_1^!=\gamma_{p_1}^*$. The expression $\gamma_{p_i}^*(x\times y)$ is also written $x\cdot_{p_i} y$ in \cite[Def. 8.1.1]{fulton} (cf. proof of Proposition \ref{p1}). The formula to be proven therefore becomes
\[ (\psi\cdot_{p_2}  \alpha)\cdot_{p_1}  \beta = (\psi\cdot_{p_1}  \beta)\cdot_{p_2}  \alpha\]
which is \cite[Prop. 8.1.1 (b)]{fulton}.
\end{proof}

\begin{rque}\label{r11} There is a much more conceptual proof by interpreting both sides as compositions of correspondences:
 we then have
\[
\langle\psi^*\alpha,  \beta\rangle=(\alpha\bullet \psi)\bullet \beta=\alpha\bullet (\psi\bullet \beta)=\langle\alpha,\psi_*\beta\rangle
\]
by the associativity of $\bullet$.
\end{rque}

\subsection{Independence from the model and functoriality}

\begin{lemma}\label{l3.4} Let $b\in B^{(1)}$ and $Z=\overline{\{b\}}$ as usual. For $\psi\in CH^*(\sX_1\times_B \sX_2)$ and $\beta\in CH_*(\sX_{1,Z})$, let
\[\psi_! \beta = (p_{2,Z})_*\delta_1^!(\psi\times \beta)\in CH_*(\sX_{2,Z}).\]
Then\\
a) $(\iota_2)_*\psi_! \beta = \psi_* ((\iota_1)_*\beta)$.\\
b) For any $\alpha\in CH^i(\sX_2)$, $\beta\in CH_{\delta+i-1}(\sX_Z)$ and $\psi\in CH^{d_2}(\sX_1\times_B \sX_2)$, we have $\langle \alpha, \psi_! \beta\rangle_b = \langle \psi^*\alpha,\beta\rangle_b$.
\end{lemma}

\begin{proof} a) Let us first draw the diagram of Cartesian squares underlying the coming computation:
\begin{equation}\label{eq18b}
\begin{CD}
\sX_{2,Z} @<p_{2,Z}<< \sX_{1,Z}\times_Z \sX_{2,Z}@>(\kappa,p_{1,Z})>> (\sX_1\times_B \sX_2)\times \sX_{1,Z}\\
@V\iota_2 VV @V\kappa VV @V 1\times \iota_1VV\\
\sX_2 @<p_2<<\sX_1\times_B \sX_2@>\gamma_{p_1}>> (\sX_1\times_B \sX_2)\times \sX_1\\
&&@Vp_1 VV @V p_1\times 1VV\\
&&\sX_1@>\delta_1>> \sX_1\times \sX_1.
\end{CD}
\end{equation}

It already explains the use of $\delta_1^!$ in the definition of $\psi_!$. Now
\begin{multline*}
(\iota_2)_*\psi_! \beta =(\iota_2)_*(p_{2,Z})_*\delta_1^!(\psi\times \beta)=(p_2)_*\kappa_*\delta_1^!(\psi\times \beta)\\
=(p_2)_*\delta_1^!(1\times\iota_1)_*(\psi\times \beta) =(p_2)_*\delta_1^!(\psi\times (\iota_1)_*\beta) \\
=:({}^t\psi)^* ((\iota_1)_*\beta)=: \psi_* ((\iota_1)_*\beta)
\end{multline*}
where the third equality follows as usual from Lemma \ref{l2.5} a).

b) First
\begin{multline*}
\alpha\cdot_{\iota_2} \psi_! \beta := \gamma_{\iota_2}^!((p_{2,Z})_*\delta_1^!(\psi\times\beta)\times \alpha)\overset{(a)}{=} \gamma_{\iota_2}^!((p_{2,Z})_*\gamma_{p_1}^!(\psi\times\beta)\times \alpha)\\
\overset{(b)}{=}(p_{2,Z})_*\gamma_{\iota_2}^!(\gamma_{p_1}\times 1)^!(\psi\times\beta\times \alpha)\\
\overset{(c)}{=}(p_{2,Z})_*\gamma_{p_2}^!(\gamma_{p_1}\times 1)^!(\psi\times\beta\times \alpha)\overset{(d)}{=}(p_{2,Z})_*\gamma_\lambda^!(\psi\times\beta\times \alpha)
\end{multline*}
where $\lambda$ is the regular embedding $\sX_1\times_B \sX_2\inj \sX_1\times \sX_2$, so that $\gamma_\lambda$ is the composition of the bottom row in the diagram of Cartesian squares
\begin{equation}\label{eq18c}\tiny
\begin{CD}
 \sX_{2,Z}@>\gamma_{\iota_2}>> \sX_{2,Z}\times \sX_2\\
 @Ap_{2,Z}AA @Ap_{2,Z}\times 1AA\\
 \sX_{1,Z}\times_Z \sX_{2,Z}@>\gamma_{p_{2,Z}\iota_2}>>(\sX_{1,Z}\times_Z \sX_{2,Z})\times \sX_2@>(\kappa,p_{1,Z})\times 1>> (\sX_1\times_B \sX_2)\times \sX_{1,Z}\times \sX_2\\
@V\kappa VV @V\kappa\times 1 VV @V 1\times \iota_1\times 1VV\\
\sX_1\times_B \sX_2@>\gamma_{p_2}>>(\sX_1\times_B \sX_2)\times \sX_2@>\gamma_{p_1}\times 1>> (\sX_1\times_B \sX_2)\times \sX_1\times \sX_2.
\end{CD}
\end{equation}

Here (a) follows from Lemma \ref{l2.5} b) applied to \eqref{eq18b}, (b) from Lemma \ref{l2.5} a), (c) from Lemma \ref{l2.5} b) again (applied twice), and (d) from Lemma \ref{l2.5} c).

Next
\begin{multline*}
\psi^*\alpha\cdot_{\iota_1} \beta:=\gamma_{\iota_1}^!(\beta\times (p_1)_*\delta_2^!(\psi \times \alpha))\overset{(a)}{=} \gamma_{\iota_1}^!(\beta\times (p_1)_*\gamma_{p_2}^!(\psi \times \alpha))\\
\overset{(b)}{=}(p_{1,Z})_*\gamma_{\iota_1}^!(1\times \gamma_{p_2})^!(\beta\times\psi \times \alpha)\\
\end{multline*}
where (a) follows from Lemma \ref{l2.5} b) applied to \eqref{eq18} and (b) follows from Lemma \ref{l2.5} a) applied to the Cartesian square
\[\begin{CD}
\sX_{1,Z}\times_Z \sX_{2,Z}@>(p_{1,Z},\kappa)>>\sX_{1,Z}\times (\sX_1\times_B \sX_2)\\
@Vp_{1,Z}VV @V1\times p_1VV\\
\sX_{1,Z}@>\gamma_{\iota_1}>> \sX_{1,Z}\times \sX_1.
\end{CD}\]

Since $f_{1,Z}p_{1,Z}=f_{2,Z}p_{2,Z}$, we are left to prove the equality
\[\gamma_\lambda^!(\psi\times\beta\times \alpha)=\gamma_{\iota_1}^!(1\times \gamma_{p_2})^!(\beta\times\psi \times \alpha).\]

For this we draw the diagram of Cartesian squares, similar to \eqref{eq18c}:
\begin{equation*}\label{eq18d}\Small
\begin{CD}
 \sX_{1,Z}@>^t \gamma_{\iota_1}>> \sX_1\times \sX_{1,Z}\\
 @Ap_{1,Z}AA @Ap_1\times 1AA\\
 \sX_{1,Z}\times_Z \sX_{2,Z}@>(\kappa,p_{1,Z})>>(\sX_{1}\times_B \sX_{2})\times \sX_{1,Z}@>\gamma_{p_2}\times 1_{\sX_{1,Z}}>> (\sX_1\times_B \sX_2)\times \sX_2\times \sX_{1,Z}\\
@V\kappa VV @V1\times \iota_1 VV @V 1\times 1\times \iota_1VV\\
\sX_1\times_B \sX_2@>\gamma_{p_1}>>(\sX_1\times_B \sX_2)\times \sX_1@>\gamma_{p_2}\times 1_{\sX_1}>> (\sX_1\times_B \sX_2)\times \sX_2\times \sX_1.
\end{CD}
\end{equation*}

Here the composition of the bottom row is $\gamma_\lambda$, up to permuting $\sX_1$ and $\sX_2$. By Lemma \ref{l2.5} b), 
$(^t \gamma_{\iota_1})^!$ and $\gamma_{p_1}^!$ both compute the refined Gysin map corresponding to the arrow $(\kappa,p_{1,Z})$, and also $(\gamma_{p_2}\times 1_{\sX_{1,Z}})^!=(\gamma_{p_2}\times 1_{\sX_{1}})^!$; we conclude by applying Lemma \ref{l2.5} c) to the bottom row once again.
\end{proof}

\begin{prop}\label{l6} Let $f_1:\sX_1\to B$, $f_2:\sX_2\to B$ be two proper morphisms with generic fibres $X_1,X_2$ of dimensions $d_1,d_2$, where $\sX_1$ and $\sX_2$ are smooth; let $r\in \Z$ and let $\gamma\in CH^{d_2+r}(X_1\times_K X_2)$ be a Chow correspondence of degree $r$. Then 
\begin{equation}\label{eq16a}
\gamma^*CH^i(X_2)^0_{f_2}\subseteq CH^{i+r}(X_1)^0_{f_1}
\end{equation} 
for any $i\ge 0$. In particular,
\begin{thlist}
\item if $r=0$, we also have $\gamma_*CH_i(X_1)^0_{f_2}\subseteq CH_i(X_2)^0_{f_1}$;
\item the group $CH^i(X)^0_f$ does not depend on $f$.
\end{thlist}
\end{prop}

\begin{proof} 
First, (i) (resp. (ii)) follows from \eqref{eq16a} by considering ${}^t\gamma$ (resp. by taking $X_1=X_2=X$, $\gamma=\Delta_X$). To prove \eqref{eq16a}, we may assume that $\gamma$ is the class of an integral cycle $\Gamma\subset X_1\times_K X_2$. 

Let $j_i:X_i\inj \sX_i$ be the corresponding immersions, and $\psi$ be the closure of $\Gamma$ in $\sX_1\times_B \sX_2$. 
By Lemma \ref{l2.1} c), 
\begin{equation}\label{eq15a}
\gamma^*\circ j_2^*=j_1^*\circ \psi^*, 
\end{equation} 
and it suffices to show that $\psi^*\alpha\in CH^{i+r}(\sX_1)^0$ for any $\alpha\in CH^i(\sX_2)^0$.  Formula \eqref{eq15a} shows that  $j_1^*(\psi^*\alpha)\in CH^{i+r}_\num(X_1)$; the other condition follows from Lemma \ref{l3.4} b).
\end{proof}

\begin{rque} If $B$ is projective, Lemma \ref{l3.4} a) is sufficient for the proof of Proposition \ref{l6} by using \eqref{eq6}, as in the proof of Proposition \ref{p2.4}.
\end{rque}

\begin{prop}\label{r4} The pairing \eqref{eq4} does not depend on the choice of  $f$ (we drop $f$ from its notation from now on). Moreover, in the situation of Proposition \ref{l6} with $r=0$, we have the identity
\begin{equation}\label{eq19a}
\langle \gamma^*\alpha,\beta\rangle = \langle \alpha,\gamma_* \beta\rangle
\end{equation}
for $(\alpha,\beta)\in CH^i(X_2)^0\times CH_{i-1}(X_1)^0$.
\end{prop}

\begin{proof}  As in the proof of Proposition \ref{l6}, the first claim follows from the second by taking $X_1=X_2=X$, $\gamma=\Delta_X$. For the second claim, we take $\gamma$ and $\psi$ as in the proof of Proposition \ref{l6}. Then \eqref{eq19a} follows from  Lemma \ref{l3.3} applied to  lifts $\tilde\alpha$ and $\tilde\beta$ of $\alpha$ and $\beta$ in $CH^i(\sX_2)^0$  and $CH^{d_1-i+1}(\sX_1)^0$ respectively.
\end{proof}

\subsection{Base change} 

\begin{prop}\label{p5} Consider a commutative diagram
\begin{equation}\label{eq10}
\begin{CD}
\sX_1@>g>>\sX_2\\
@V{f_1}VV @V{f_2}VV\\
B_1@>\bar g>> B_2
\end{CD}
\end{equation}
where $f_1,f_2$ satisfy the hypotheses of \S \ref{s1.1}, $\bar g$ is finite surjective and $g$ proper; we assume that the diagram of generic fibres:
\[\begin{CD}
X_1@>g'>>X_2\\
@V{f'_1}VV @V{f'_2}VV\\
\eta_1@>\bar g'>> \eta_2
\end{CD}\]
is Cartesian (in particular, $g$ is generically finite). Then, for all $i\ge 0$, one has 
\begin{thlist}
\item $g^*CH^i(\sX_2)^0\subseteq CH^i(\sX_1)^0$, hence ${g'}^*CH^i(X_2)^0\allowbreak\subseteq CH^i(X_1)^0$.
\item $g_*CH^i(\sX_1)^0\subseteq CH^i(\sX_2)^0$, hence $g'_*CH^i(X_1)^0\allowbreak\subseteq CH^i(X_2)^0$. 
\item $(g^*)^{-1}CH^i(\sX_1)^0=CH^i(\sX_2)^0$.
\item One has the identities 
\begin{align}
\bar g_*\langle {g'}^*\alpha,\beta'\rangle &= \langle \alpha,g'_*\beta'\rangle\label{eq20a}\\
\langle {g'}^*\alpha,{g'}^*\beta\rangle &= \bar g^*\langle \alpha,\beta\rangle\label{eq20}
\end{align}
for any $i\ge 0$ and any $(\alpha,\beta,\beta')\in CH^i(X_2)^0\times CH^{d+1-i}(X_2)^0\times CH^{d+1-i}(X_1)^0$.
\end{thlist}
\end{prop}

\begin{proof} (i) Write $j_i:X_i\inj \sX_i$ for the inclusions. Let $\alpha\in CH^i(\sX_2)^0$: then $j_1^*g^*\alpha={g'}^*j_2^*\alpha\in CH^i_\num(X_1)$.
Next, let $b\in B_1^{(1)}$ and $Z=\overline{\{b\}}$. Let $\beta\in CH_{\delta+i-1}(\sX_{1,Z})$, $f_{1,Z}:\sX_{1,Z}\to Z$ be the restriction of $f_1$ and $\iota_1:\sX_{1,Z}\inj \sX_1$ be the closed immersion : we need to prove that $(f_{1,Z})_*(g^* \alpha\cdot_{\iota_1} \beta)=0$.
Let $T=\bar g(Z)$ and $\bar h:Z\to T$ be the (finite surjective) projection: it suffices to show that $\bar h_*(f_{1,Z})_*(g^* \alpha\cdot_{\iota_1} \beta)=0\in CH^0(T)$. This follows from the computation
\begin{multline*}
0\overset{(a)}{=}(f_{2,T})_*(\alpha\cdot_{\iota_2} h_*\beta)=(f_{2,T})_*\gamma_{\iota_2}^!(h_*\beta\times \alpha)\\
\overset{(b)}{=}(f_{2,T})_*h_*\gamma_{g\iota_1}^!(\beta\times \alpha)=\bar h_*(f_{1,Z})_*\gamma_{g\iota_1}^!(\beta\times \alpha)\\\overset{(c)}{=}\bar h_*(f_{1,Z})_*\gamma_{\iota_1}^!(1\times g)^!(\beta\times \alpha)=\bar h_*(f_{1,Z})_*(g^*\alpha\cdot_{\iota_2}\beta)
\end{multline*}
where $h:\sX_{1,Z}\to \sX_{2,T}$ is the restriction of $g$ and $\iota_2$ is the inclusion $\sX_{2,T}\inj \sX_2$, in which (a) is by hypothesis, (b) follows from Lemma \ref{l2.5}, and (c) follows from  \cite[Prop. 8.1.1 (a)]{fulton} (see comment in op. cit., mid p. 134).

(ii) The inclusion $j_2^*g_*CH^i(\sX_1)^0\subseteq CH^i_\num(X_2)$ is obtained this time from the identity $j_2^*g_*=g'_*j_1^*$.
Next, let $b\in B_2^{(1)}$ $Z=\overline{\{b\}}$ and $\iota_2:\sX_{2,Z}\inj \sX_2$, $f_{2,Z}:\sX_{2,Z}\to Z$ be the inclusion and the projection. 
Let $\alpha\in CH^i(\sX_1)^0$ and 
$\beta\in CH_{\delta+i-1}(\sX_{2,Z})$: we need to prove that $(f_{2,Z})_*(g_*\alpha \cdot_{\iota_2} \beta)=0\in CH^0(Z)$. 

Let $T=\bar g^{-1}(Z)$. Then $\sX_{1,T}\iso \sX_1\times_{\sX_2} \sX_{2,Z}$, hence refined Gysin morphisms $g^!:CH_j(\sX_{2,Z})\to CH_j(\sX_{1,T})$.  By Lemma \ref{l2.4}, we have $(f_{1,T})_*(\alpha \cdot_{\iota_1} g^!\beta)=0$ where $\iota_1$ is the inclusion $\sX_{1,T}\inj \sX_1$. The commutative square
\[\begin{CD}
\sX_{1,T}@>h>> \sX_{2,Z}\\
@Vf_{1,T}VV @Vf_{2,Z}VV\\
T@>\bar h>> Z,
\end{CD}\]
where $h$ and $\bar h$ are the restrictions of $g$ and $\bar g$, gives the identity of push-forwards
\[\bar h_* (f_{1,T})_* =(f_{2,Z})_*h_*. \]

Therefore, it suffices to prove the identity (projection formula)
\begin{equation}\label{eq2.2}
 g_*\alpha \cdot_{\iota_2} \beta = h_*(\alpha\cdot_{\iota_1} g^!\beta).
\end{equation}

For this, consider the commutative diagram of Cartesian squares

\[\begin{CD}
\sX_1@>\delta_{\sX_1}>> \sX_1\times \sX_1@>g\times 1>> \sX_2\times \sX_1\\
@A\iota_1 AA @A\iota_1\times 1AA @A\iota_2\times 1AA \\
\sX_{1,T}@>\gamma_{\iota_1}>> \sX_{1,T}\times \sX_1@>h\times 1>> \sX_{2,Z}\times \sX_1\\
@VhVV&& @V1\times g VV\\
\sX_{2,Z}&@>\gamma_{\iota_2}>>& \sX_{2,Z}\times \sX_2\\
@V\iota_2 VV&& @V\iota_2\times 1 VV\\
\sX_{2}&@>\delta_{\sX_2}>>& \sX_{2}\times \sX_2.
\end{CD}\]

Applying Lemma \ref{l2.5} to the two bottom squares yields first
\[g_*\alpha \cdot_{\iota_2} \beta := \gamma_{\iota_2}^!(1\times g)_*(\beta\times \alpha) = h_*\gamma_{\iota_2}^!(\beta\times \alpha) = h_*\delta_{\sX_2}^!(\beta\times \alpha).\]

We are now left to show the identity
\[\delta_{\sX_2}^!(\beta\times \alpha)=\alpha\cdot_{\iota_1} g^!\beta:= \gamma_{\iota_1}^!(g\times 1)^!(\beta\times \alpha)\]
where the right hand side stems from the top part of the diagram (with vertical arrows pointing  upwards). But $\gamma_{\iota_1}^!(g\times 1)^!(\beta\times \alpha)=\delta_{\sX_1}^!(g\times 1)^!(\beta\times \alpha)$  and $\delta_{\sX_1}^!(g\times 1)^!(\beta\times \alpha)=[(g\times 1)\delta_{\sX_1}]^!(\beta\times \alpha)={}^t\gamma_g^!(\beta\times \alpha)$, both by Lemma \ref{l2.5}. Here, $^t\gamma$ denotes the transpose of a graph (graph composed with the switch of factors). Finally, ${}^t\gamma_g^!(\beta\times \alpha)=\delta_{\sX_2}^!(\beta\times \alpha)$ by applying once again Lemma \ref{l2.5} b) to the diagram of Cartesian squares
\[\begin{CD}
\sX_{1,T}@>(h,\iota_1)>> \sX_{2,Z}\times \sX_1\\
@V\iota_1 VV @V\iota_2\times 1VV\\
\sX_1@>{}^t\gamma_g>> \sX_2\times \sX_1\\
@Vg VV @V1\times gVV\\
\sX_2@>\delta_{\sX_2}>> \sX_2\times \sX_2.
\end{CD}\]

(iii) follows from (i) and (ii) by the projection formula $g_*{g}^*=\deg(g)$ (generic degree), and Lemma \ref{l3}.

(iv) follows from the special case $B_1=B_2$, $\bar g=1_B$ in (i) or (ii).

In (iv), the identities can be checked on the level of $\sX_1$ and $\sX_2$. The first, \eqref{eq20a}, is an easy consequence of the projection formula. Let us prove \eqref{eq20}. The diagram of Cartesian squares
\[\begin{CD}
\sX_1@>(g,f_1)^B>> \sX_2\times_{B_2} B_1@>1\times_{B_2} \bar g>> \sX_2\\
@V||VV @V\text{inj}VV @V(1,f_2) VV\\
\sX_1@>(g,f_1)>>\sX_2\times_k B_1@>1\times \bar g>> \sX_2\times_k B_2
\end{CD}\]
together with Lemma \ref{l2.5} b) and c) gives a factorisation of $g^*$ into a composition of refined Gysin morphisms
\begin{equation}\label{eq21}
g^* =  (g,f_1)^!(1\times \bar g)^!.
\end{equation}

Next, \cite[Ex. 8.1.7]{fulton} applied to the left square with $x=[\sX_1]$ and $y= (1\times \bar g)^!z$ for some $z\in CH^*(\sX_2)$ yields via \cite[Prop. 8.1.2 (b)]{fulton} the identity
\begin{equation} \label{eq22}
(g,f_1)^B_*(g,f_1)^! y =  (g,f_1)^B_*[\sX_1] \cdot y=y;
\end{equation}
indeed, $(g,f_1)^B$ maps $\sX_1$ birationally onto an irreducible component of $\sX_2\times_{B_2} B_1$, and the other irreducible components have support away from $\eta_2$, hence have smaller dimensions. Taking $z=\alpha\cdot \beta$ for $(\alpha,\beta)\in CH^i(\sX_2)^0\times CH^{d+1-i}(\sX_2)^0$, we get
\begin{multline*}
\langle g^*\alpha,g^*\beta\rangle = (f_1)_* (g^*\alpha\cdot g^*\beta) = (f_1)_* g^*(\alpha\cdot \beta)\\
\overset{\eqref{eq21}}{=} (f_2\times 1)_*(g,f_1)^B_* (g,f_1)^!(1\times \bar g)^!(\alpha\cdot \beta)\\
\overset{ \eqref{eq22}}{=} (f_2\times 1)_*(1\times \bar g)^!(\alpha\cdot \beta)\\ 
= \bar g^* (f_2)_*(\alpha\cdot \beta) = \bar g^* \langle \alpha,\beta \rangle
\end{multline*}
where the last but one equality follows once again from Lemma \ref{l2.5}. This readily implies \eqref{eq20}. 
\end{proof}

\begin{rque}\label{r5} In Proposition \ref{p5}, suppose that $\bar g$ is only an alteration. I cannot prove (i). On the other hand, (ii) holds with the same proof, as well as (iv) for $(\alpha,\beta)\in CH^i(X_2)^0\times CH^{d+1-i}(X_2)^0$ such that $(g^*\alpha,g^*\beta)\in CH^i(X_1)^0\times CH^{d+1-i}(X_1)^0$. This is not very important, in view of Remark \ref{r1} b) (see proof of Proposition \ref{p2.2}).
\end{rque}

\subsection{Structure of $CH^i(X)/CH^i(X)^0$}

\begin{prop}\label{p2.2}  
The groups $CH^i(\sX)/CH^i(\sX)^0$ and $CH^i(X)/CH^i(X)^0$ are finitely generated.
\end{prop}

\begin{proof} It suffices to show the first claim. We proceed in several steps. 

1) Suppose $B'$ is an open subset of $B$, let $\sX'=\sX\times_B B'$ and let $\lambda:\sX'\to \sX$ be the corresponding open immersion. Then $CH^i(\sX)^0\subseteq (\lambda^*)^{-1}CH^i(\sX')^0$, with equality if $B-B'$ has codimension $\ge 2$. Therefore the claim for $\sX$ implies the claim for $\sX'$, and conversely in the latter case.

2) $B$ is projective: this follows from Proposition \ref{p2.4}.

3) In general, let $\bar B$ be a compactification of $B$ and $\bar \sX\by{\bar f}\bar B$ a projective morphism extending $f$ (in the sense that $\sX=\bar \sX\times_{\bar B} B$). 

\begin{itemize}
\item By \cite[Th. 4.1]{dJ}, alter $\bar B$ into a smooth projective $k$-variety $\bar B_1$.
\item Let $K_1=k(\bar B_1)$ (a finite extension of $K$), and let $\bar \sX'$ be the closure of $X\otimes_K K_1$ in $\bar \sX\times_{\bar B} \bar B_1$. Again by \cite[Th. 4.1]{dJ}, alter $\bar \sX'$ into a smooth projective $k$-variety $\bar \sX_1$. We are now in the situation of 2).
\item Let $B_1=B\times _{\bar B} \bar B_1$ and $\sX_1=B_1\times_{\bar B_1} \bar \sX_1$. 
\item By Remark \ref{r1} b), the alteration $B_1\to B$ becomes flat, hence finite, after removing from $B$ a closed subset $F$ of codimension $\ge 2$. Let $B'=B-F$ and $\sX'$, $B'_1$, $\sX'_1$ be the corresponding base changes of $\sX$, $B_1$ and $\sX_1$.
\end{itemize}

By 2), the claim is true for $\bar \sX_1$; therefore it is also true for $\sX'_1$ by 1). By Proposition \ref{p5} (i) and (ii), the projection $\sX'_1\to \sX'$ induces maps between $CH^i(\sX')/CH^i(\sX')^0$ and $CH^i(\sX'_1)/CH^i(\sX'_1)^0$, whose composition is multiplication by $[K_1:K]$. Since $CH^i(\sX')/CH^i(\sX')^0$ is torsion-free by Lemma \ref{l3}, it is finitely generated, and so is $CH^i(\sX)/CH^i(\sX)^0$ by reapplying 1).
\end{proof}

\begin{rque}\label{r2.3} Proposition \ref{p2.1} gives a more direct proof of Proposition \ref{p2.2} in characteristic $0$, by the comparison theorem between Betti and $l$-adic cohomology. 
\end{rque}

\subsection{A vanishing result} 

Let $l$ be a prime number invertible in $k$. For any smooth $k$-variety $V$, there are cycle class maps with values in Jannsen's continuous étale cohomology
\[\cl^i:CH^i(V)\to H^{2i}_\cont(V,\Z_l(i))\]
which are compatible with pull-backs, push-forwards and products \cite[(3.25) and (6.14)]{jannsen}.\footnote{Strangely, \cite[(3.25)]{jannsen} only mentions push-forwards for closed immersions, but the case of a general proper morphism is proven in the same way.} 

\begin{lemma}\label{l2.3} Suppose $k$ finitely generated. Then the composition of $\cl^1$ with the projection $H^{2}_\cont(V,\Z_l(1))\to H^{2}_\cont(V,\Z_l(1))/H^{2}_\cont(k,\Z_l(1))$ has finite kernel.
\end{lemma}

\begin{proof} By construction of $\cl^i$, there is a commutative diagram
\[\begin{CD}
CH^1(V)@>\alpha >> CH^1(V)^\wedge\\
@V\cl^1 VV @V(\cl^1)^\wedge VV\\
H^2_\cont(V,\Z_l(1))@>>> \lim H^2(V,\mu_l^n)
\end{CD}\]
where the bottom map is part of the Milnor exact sequence of \cite[(3.16)]{jannsen} and $CH^1(V)^\wedge$ is the $l$-adic completion of $CH^1(V)$. The Kummer exact sequences imply the injectivity of $(\cl^1)^\wedge$. Since $k$ is finitely generated, $CH^1(V)$ is a finitely generated abelian group, which implies that $\alpha$ has finite kernel of order prime to $l$. Hence the same holds for $\cl^1$.

On the other hand, the choice of a $0$-cycle of nonzero degree on $V$ (e.g. a closed point), plus transfer, provide a map $\rho: H^2_\cont(V,\Z_l(1))\to H^2_\cont(k,\Z_l(1))$ such that the composition 
\[H^2_\cont(k,\Z_l(1))\to H^2_\cont(V,\Z_l(1))\by{\rho} H^2_\cont(k,\Z_l(1))\] 
is multiplication by some integer $m>0$. Since $CH^1(k)=0$, the naturality of the cycle class map implies that $\rho\circ \cl^1=0$. Hence the lemma.
\end{proof}

The following proposition will be used in the proof of Proposition \ref{p5.4}.

\begin{prop}\label{p5.1} Let $(\alpha,\beta)\in CH^i(\sX)\times CH^{d+1-i}(\sX)$. Consider the pairing \eqref{eq2}.
As in \S \ref{s2.3}, let $CH^i_l(\sX)$ be the kernel of the \emph{geometric} cycle class map. If $(\alpha,\beta)\in CH^i_l(\sX)\times CH^{d+1-i}_l(\sX)$, then  $\langle \alpha,\beta\rangle$ is torsion. 
\end{prop}

\begin{proof} We may assume $k$ to be the perfect closure of a finitely generated field. We use the spectral sequences  of \cite[Th. (3.3)]{jannsen}
\[E_2^{p,q} = H^p_\cont(k,H^q(V_{\bar k},\Z_l(n))\Rightarrow H^{p+q}_\cont(V,\Z_l(n)).\]

They are compatible with the action of correspondences, in particular with products and push-forwards. Thus, if $F^\bullet H_\cont$ is the filtration on $H_\cont$ induced by the spectral sequence, we have
\begin{multline*}
\cl^1(f_*(\alpha\cdot \beta))=f_*\cl^{d+1}(\alpha\cdot \beta) = f_*(\cl^i(\alpha)\cup \cl^{d+1-i}(\beta))\\
 \in F^2H^2_\cont(B,\Z_l(1))=\IM(H^2_\cont(k,\Z_l(1))\to H^2_\cont(B,\Z_l(1)))
\end{multline*}
if $(\alpha,\beta)\in CH^i_l(\sX)\times CH^{d+1-i}_l(\sX)$. We conclude by Lemma \ref{l2.3}.
\end{proof}

\begin{qn}\label{q1} When $B$ is projective, can one prove Proposition \ref{p5.1} with $CH_l$ replaced by $CH_\num$, without assuming the standard conjectures?
\end{qn}

\subsection{Local height pairing}

In this context, there is not much to say. Let $f$ be as in \S \ref{s1.1}. Let $C_1\in \sZ^i(X)$, $C_2\in \sZ^{d+1-i}(X)$ be two integral cycles with disjoint supports. Let $\sC_i$ be the closure of $C_i$ in $\sX$; then $\sC_1\times_\sX \sC_2$ has support in $\sX_Z$ for some proper closed subset $Z$ of $B$, whence a refined intersection product \cite[\S 8.1]{fulton}:
\[\sC_1\cdot \sC_2\in CH_{\delta-1}(\sX_Z).\]

Given the isomorphism
\[CH_{\delta-1}(Z)\iso \bigoplus_{b\in Z\cap B^{(1)}} \Z, \]
the class $(f_Z)_*(\sC_1\cdot \sC_2)$ defines a divisor on $B$, whose class in $\Pic(B)=CH^1(B)$ is obviously $\langle \sC_1,\sC_2\rangle$ (cf. \cite[Lemma 2.0.1]{beilinson}). One may then extend by bilinearity and get an expression of $\langle,\rangle$ as the class of a divisor.

We leave it to the interested reader to refine Lemma \ref{l3.3} to this local height pairing in the style of \cite[(A.2)]{bloch}.

\section{Extension to the general case}\label{s4a}

Let $X$ be regular, connected and proper over $K$. In the previous section, we defined subgroups $CH^i(X)^0\subset CH^i(X)$ and pairings \eqref{eq0} assuming the existence of a $k$-smooth model $\sX$ of $X$, proper over $B$.

\subsection{Characteristic $0$}\label{s4.1}

\begin{prop}\label{p3.1} Assuming resolution of singularities à la Hironaka, a smooth model always exists. This is the case in particular if $\car k=0$, or if $d+\delta\le 3$ \cite{cp}.
\end{prop}

\begin{proof} Start from an integral proper model $f:\sX\to B$ of $X/K$. Let $U\subseteq \sX$ be the regular locus of $\sX/k$: it is open (EGA IV$_2$, Cor. 6.12.6) and since $X$ is regular, we have $X\subset U$. By hypothesis, we may find $\sX_1$ regular over $k$ and a projective morphism $\pi:\sX_1\to \sX$ such that $\pi_{| \pi^{-1}(U)}:\pi^{-1}(U)\to U$ is an isomorphism. Then the immersion $X\inj \sX$ lifts to $X\inj \sX_1$, and $\sX_1$ is the desired smooth model of $X$ (since $k$ is assumed to be perfect).
\end{proof}

\subsection{Positive characteristic}\label{s4.2} Here we cannot directly use de Jong's theorem \cite{dJ} to replace Hironaka resolution, because there is no control in this theorem on the centre of the alteration. Instead we must proceed more carefully.

\begin{defn}\label{d3.2} Let  $X$ be an integral proper $K$-scheme.\\
a) $X$ is \emph{good} (relatively to $B$) if it admits a $k$-regular proper model $\sX\by{f}B$. (In particular, $X$ is then regular.)\\
b)  A $K$-morphism $\pi:X_1\to X$ is \emph{admissible} if $X_1$ is good.\\
c) We set
\[CH^i(X)^0 =  \{\alpha\in CH^i(X)\mid \ \pi^* \alpha\in CH^i(X_1)^{0}\ \forall\ \pi \text{ admissible}\}.\]
\end{defn}

\begin{lemma}\label{l4.3} a) For any $X$ as in Definition \ref{d3.2}, admissible alterations exist; in particular $CH^i(X)^0\neq \emptyset$.\\
b) If $X$ is good, $CH^i(X)^0$ agrees with the subgroup of Proposition \ref{l6} (ii).\\
c) Given two admissible morphisms $\pi_i:X_i\to X$, there exists an admissible morphism $\pi_3:X_3\to X$ factoring through $\pi_1$ and $\pi_2$.\\
d) If $X$ is regular, we have $CH^i(X)^0\subseteq CH^i_\num(X)$.
\end{lemma}

\begin{proof} a) follows from \cite[Th. 4.1]{dJ} applied to a (not necessarily smooth) model. b) follows from Proposition \ref{l6}. c) Let $\sX,\sX_1, \sX_2$ be $B$-proper models of $X,X_1$ and $X_2$. Taking the graphs of the rational maps $\pi_i:\sX_i\tto \sX$, we may assume these to be $B$-morphisms.  Applying \cite[Th. 4.1]{dJ} again to an irreducible component of $\sX_1\times_\sX\sX_2$ dominant over $B$, we get a $k$-smooth $\sX$-scheme $\sX_3$, projective over $B$ and mapping to $\sX_1$ and $\sX_2$, whose generic fibre $X_3$ maps to $X_1$ and $X_2$ over $X$. 

 For d), let  $\alpha\in CH^i(X)^0$ and $\beta\in CH^{d-i}(X)$. Choose an admissible $\pi$. Writing $[,]$ for the intersection product, we have $[\pi_\eta^*\alpha,\pi_\eta^*\beta]=0$ by definition of $CH^i(X_1)^0$, hence $[\alpha,\beta]=0$. 
\end{proof}

To go further, we need to invert $p$ in characteristic $p$; this is the object of the next subsections.

\subsection{The category $\Ab\otimes R$, where $R$ is a subring of $\Q$} \label{s3.2} (See also \cite[App. B]{BVK}). This category has two equivalent descriptions:

\begin{itemize}
\item It is the localisation of the category $\Ab$ of abelian groups with respect to the Serre subcategory of abelian groups killed by some integer invertible in $A$; in particular, $\Ab\otimes R$ is abelian and the localisation functor $\Ab\to \Ab\otimes R$ is exact.
\item It has the same objects as $\Ab$, while morphisms are those of $\Ab$ tensored with $A$.
\end{itemize}

If $R=\Z[1/p]$, we shall abbreviate $\Ab\otimes R$ to $\Ab[1/p]$.

\begin{lemma} \label{l3.1} The tensor product of $\Ab$ induces a tensor structure on $\Ab\otimes R$, still denoted by $\otimes$.
\end{lemma}

(This allows us to talk of a ``pairing in $\Ab\otimes R$''.)

\begin{proof} It suffices to show that, if $f\in\Ab(A,B)$ becomes invertible in $\Ab\otimes R$ (i.e. $\Ker f$, $\Coker f$ have $p$-power exponent), the same holds for $f\otimes 1_C$ for any $C\in \Ab$. By considering the image of $f$, we may treat separately the cases where $f$ is injective and $f$ is surjective. Both hold because, if $G\in \Ab$ has $p$-power exponent, so do $G\otimes C$ and $\Tor(G,C)$ for any $C\in \Ab$.
\end{proof}

\begin{rques}\label{r10} a) Let $A,B$ be two abelian groups. By definition, a morphism in $(\Ab\otimes R)(A,B)=\colim_{N\ne 0} \Ab(A,B)$ is represented by a pair $(\phi,N)$ with $\phi:A\to B$ and $N$ an integer invertible in $R$; two pairs $(\phi_1,N_1)$ and $(\phi_2,N_2)$ are equivalent if there exist two such integers $d_1,d_2$ such that $d_1N_1=d_2N_2=:N_3$ and $(d_1\phi_1,N_3)=(d_2\phi_2,N_3)$. We get a well-defined homomorphism $\rho:(\Ab\otimes R)(A,B)\to \Ab(A,B\otimes R)$ by sending a pair $(\phi,N)$ to $\psi ;= N^{-1}\phi$; its image is contained in the subgroup formed of those homomorphisms $\psi:A\to B\otimes R$ such that $\psi(A)\subseteq N^{-1} \bar B$ for some $N\ne 0$, with $\bar B=B/\text{torsion}$. If $B$ is torsion-free, $\rho$ is injective with the above image.\\
b) In any category, the commutativity of a diagram (i.e. the equality of two arrows) is equivalent to the commutativity of  a family of diagrams of sets, thanks to Yoneda's lemma. In the category of modules over a ring $R$, one can test such commutativity on elements, because the $R$-module $R$ is a generator.  
  
In the sequel, we shall extend identities such as   \eqref{eq19a}, \eqref{eq20a} and \eqref{eq20} to $\Ab\otimes R$.  However this category is not Grothendieck (note that abelian groups with finite exponent are not closed under infinite direct sums), so reasoning with ``elements'' is abusive. Writing out the above identities as commutative diagrams in $\Ab$ is straightforward, but cumbersome. (For example, \eqref{eq19a} means that two homomorphisms from $CH^{d_2}(X_1\times_K X_2)\otimes CH^i(X_2)^0\otimes CH_{i-1}(X_1)^0$ to $CH^1(B)$ agree.) We shall therefore sometimes make the abuse of talking of such identities in $\Ab\otimes R$ when we mean the corresponding commutative diagrams.
\end{rques}
 
 In Theorem \ref{p3}, we shall use a local-to-global result for these localisations (Corollary \ref{t4.1} below). 
 
 \begin{thm}\label{t4.0} Let $H$ be a module over an integral domain $R$ with quotient field $Q$. 
Suppose given, for each maximal ideal $\fm\subset R$, an element $f_\fm\in H_\fm$, all of which become equal in $Q\otimes_R H$. Then there exists at most one element $f\in H$ which becomes equal to $f_\fm$ in $H_\fm$ for every $\fm$; $f$ exists provided 
\begin{thlist}
\item $H$ is torsion free, or
\item $R$ is Noetherian and $S=\Supp(M_\tors)$ is a finite set of maximal ideals.
\end{thlist}
 \end{thm}
 
(Counterexample without Hypothesis (ii): $R=\Z$, $H=\bigoplus_\fm \Z/\fm$, $f_\fm=1_\fm$.)
 
\begin{proof} \emph{Uniqueness.} Let $f,f'$ verifying the condition. Then $f$ and $f'$ become equal in $H_\fm$ for all $\fm$. This means that, for every $\fm$, there exists $M_\fm\in R-\fm$ such that $M_\fm(f-f')=0$. Since the $M_\fm$ generate $R$ as an ideal, we get $f=f'$.

\emph{Existence.}  
We may write $f_\fm=r_\fm^{-1}\tilde f_\fm$  with $\tilde f_\fm\in H$ and $r_\fm\in R- \fm$; again, the $r_\fm$'s generate the unit ideal of $R$. In case (i), if $g\in Q\otimes_R H$ is the common value of the $f_\fm$, then $r_\fm g\in H$ for  all $\fm$; if $(a_\fm)$ is a family of elements of $R$ with finite support such that $\sum a_\fm r_\fm=1$, then $g=\sum a_\fm r_\fm g\in H$.

In case (ii), write $T=H_\tors$ for notational simplicity.  Considering $H/T$, we find $f_0$ such that $1_\fm \otimes f_0 -f_\fm$ is torsion for all $\fm$, hence is $0$ for $\fm\notin S$. 

\begin{claim} The monomorphism $T\mapsto \prod_{\fm\in S} T_\fm$ is surjective.
\end{claim}

\begin{proof} For each $\fm\in S$, let $T^\fm=\Ker(T\to \prod_{\fm'\ne \fm} T_{\fm'})$: we must show that $T=\sum T^\fm$. Let $t\in T$; by assumption, the radical of $\Ann(t)$ (the annihilator of $t$) is of the form $\prod_{\fm\in S'}\fm$ for a subset $S'$ of $S$. By \cite[IV.2.5, Prop. 9]{bbki}, $R(t)=R/\Ann(t)$ is Artinian, hence $R(t) \iso \prod_{\fm\in S'} R(t)_\fm$ (ibid., Cor. 1); equivalently, $Rt\iso \prod_{\fm\in S'} (Rt)_\fm$, which shows that $t\in \sum T^\fm$.
\end{proof}

Coming back to the proof of Case (ii), the claim yields an element $t\in T$ such that $t_\fm=1_\fm \otimes f_0 -f_\fm$ for all $\fm\in S$; then $f=f_0-t$ yields the desired element.
 \end{proof}

 \begin{cor}\label{t4.1} Let $A,B\in \Ab$ and $R$ be a subring of $\Q$. Suppose given, for each prime number $l$ not invertible in $R$, a morphism $f_l:A\to B$ in $\Ab\otimes\Z_{(l)}$, all of which become equal in $\Ab\otimes \Q$. Then there exists at most one morphism $f:A\to B$ in $\Ab\otimes R$ which becomes equal to $f_l$ in $\Ab\otimes\Z_{(l)}$ for every $l$; $f$ exists provided $B$ is $l$-torsion free for almost all $l$ not invertible in $R$.
 \end{cor}

\begin{proof} Apply Theorem \ref{t4.0} to  $H=\Hom(A,B)\otimes R$, noting that the hypothesis on $B$ implies the hypothesis on $H$.
\end{proof}

\subsection{$p$-covers}

\begin{defn}\label{d3.3a} Let  $X$ be an integral proper $K$-scheme. A \emph{$p$-cover} of $X$ is a finite family $(\pi_l:X_l\to X)$, indexed by prime numbers $l\ne p$ and such that
\begin{thlist}
\item for each $l$,  $\pi_l$ is an admissible alteration of generic degree $d_l$ prime to $l$;
\item $\gcd_l(d_l)$ is a power of $p$.
\end{thlist}
\end{defn}

\begin{prop}\label{p3.3} a) $p$-covers exist. \\
b) Given two $p$-covers $(\pi_l)$, $(\pi'_l)$, there exists a third $p$-cover $(\pi''_l)$ such that, for each $l$, $\pi''_l$ factors through $\pi_l$ and $\pi'_l$.\\
c) Given a $p$-cover $(\pi_l)$ and an admissible morphism $f_1:X_1\to X$, there exists a $p$-cover $(\pi_{1,l})$ of $X_1$ such that the composition $X_{1,l}\to X_1\to X$ factors through $X_l$ for each $l$.
\end{prop}

\begin{proof} a) We use Gabber's refinement of de Jong's alteration theorem \cite[Exp. X, th. 2.1]{gabber}: given a model $\sX$ of $X$ and a prime number $l\ne p$, we may find an alteration $\sX_l\to \sX$ with $\sX_l$ regular (hence smooth over $k$) and of generic degree $d_l$ prime to $l$; the induced alteration $\pi_l:X_l\to X$ is then admissible of generic degree $d_l$.   Considering the other prime divisors of $d_l$ different from $p$, we may find a finite number of $l$'s and $\pi_l$'s such that the gcd of the $d_l$ is a power of $p$. 

b) and c) are proven similarly to a).
\end{proof}

\subsection{The refined height pairing (characteristic $p$)}

\begin{defn}\label{d3.3} We set
\[CH^i(X)^{[0]}= \{\alpha\in CH^i(X)\mid \ \exists s\ge 0: p^s \alpha\in CH^i(X)^0\}.\]
\end{defn}

\begin{prop}\label{p3.4}
a) If $X$ is regular, $CH^i(X)/CH^i(X)^0$ is an extension of a finitely generated abelian group by a torsion group of $p$-power exponent, and $CH^i(X)/CH^i(X)^{[0]}$ is finitely generated with prime-to-$p$ torsion.\\
b) Let $(\pi_l)$ be a $p$-cover of $X$, and let $\alpha\in CH^i(X)$. Then $\alpha\in CH^i(X)^{[0]}$ if and only if $\pi_l^*\alpha\in CH^i(X_l)^{[0]}$ for each $l$.\\
c) Propositions \ref{l6} and \ref{p5} (i), (ii), (iii) extend to all regular $X$'s after replacing $CH^i(X)^0$ by $CH^i(X)^{[0]}$.
\end{prop}

\begin{proof} 
a) Given a $p$-cover $(\pi_l)$, since $(\pi_l)_*\pi_l^*$ is multiplication by $d_l$ for each $l$, $\Ker(CH^i(X)/CH^i(X)^0\to \prod_l CH^i(X_l)/CH^i(X_l)^0)$ is killed by a power of $p$, say $p^s$, and the first claim follows from Proposition \ref{p2.2}. The second follows by definition of $CH^i(X)^{[0]}$.

b) The condition is necessary by definition; the converse follows from Proposition \ref{p3.3} c), as in a). 

c) Let $X_1,X_2,\gamma$ be as in Proposition \ref{l6}. To prove \eqref{eq16a}, we must show that $\pi^*\gamma^*CH^i(X_2)^0\subseteq CH^{i+r}(X'_1)^0$ for any admissible $\pi:X'_1\to X_1$; replacing $\gamma$ by $\gamma\circ \pi$, we may assume that $X_1$ is good and $\pi=1_{X_1}$. Choose a $p$-cover $(\pi_l)$ of $X_2$. For $\alpha\in CH^i(X_2)^{[0]}$, we have $\pi_l^*\alpha\in  CH^i(X_l)^{[0]}$, hence
\[d_l\gamma^*\alpha = (\gamma^*(\pi_l)_*)\pi_l^*\alpha\in CH^{i+r}(X_1)^{[0]}\]
for all $l$ thanks to Proposition \ref{l6}, hence $p^s\gamma^*\alpha\in CH^{i+r}(X_1)^{[0]}$ and finally $\gamma^*\alpha\in CH^{i+r}(X_1)^{[0]}$ as desired. The cases in Proposition \ref{p5} are treated similarly.
\end{proof}

\begin{lemma}\label{l4.2} Let $\pi:X_1\to X$ be an admissible alteration, of generic degree $d$ prime to $l$, where $l\ne p$. Then the morphism in $\Ab\otimes \Z_{(l)}$
\begin{multline*}
\langle, \rangle_{(l)}:CH^i(X)^{[0]}\otimes CH^{d+1-i}(X)^{[0]}\by{(\pi^*\otimes \pi^*)}\\CH^i(X_1)^{[0]}\otimes CH^{d+1-i}(X_1)^{[0]}
\by{d^{-1}\langle,\rangle} CH^1(B)
\end{multline*}
does not depend on the choice of $\pi$, and coincides with $\langle,\rangle$ if $X$ is good. For two prime numbers $l,l'\ne p$, we have $\langle, \rangle_{(l)}=\langle, \rangle_{(l')}$ in $\Ab\otimes \otimes\Q$.
\end{lemma}

\begin{proof} Let $\pi':X'_1\to X$ another such alteration, with generic degree $d'$. By Proposition \ref{p3.3} c) applied to an irreducible component of $X_1\times X X'_1$ dominating $X$, we can find admissible alterations $X''_1\by{\rho} X_1$, $X''_1\by{\rho'} X'_1$ of generic degrees $\delta,\delta'$ such that $\pi\rho=\pi'\rho'$, hence $\delta d = \delta'd'$. Using elements to clarify the argument, we have for $(\alpha,\beta)\in CH^i(X)^{[0]}\times CH^{d+1-i}(X)^{[0]}$
\begin{multline*}
d^{-1}\langle \pi^*\alpha,\pi^*\beta\rangle=d^{-1}\delta^{-1}\langle \rho^*\pi^*\alpha,\rho^*\pi^*\beta\rangle\\ 
 ={d'}^{-1}{\delta'}^{-1}\langle {\rho'}^*{\pi'}^*\alpha,{\rho'}^*{\pi'}^*\beta\rangle={d'}^{-1}\langle {\pi'}^*\alpha,{\pi'}^*\beta\rangle
\end{multline*}
where we used \eqref{eq19a} and the identities $\rho_*\rho^*=\delta$, $\rho'_*{\rho'}^*=\delta'$. The second claim follows by taking $\pi=1_X$. For the third claim, we argue similarly by using an admissible alteration covering two admissible alterations of generic degrees prime to $l$ and $l'$.
\end{proof}

 \begin{thm}\label{p3} a) There exists a unique pairing
 \begin{equation}\label{eq9}
\langle,\rangle:CH^i(X)^{[0]}\otimes CH^{d+1-i}(X)^{[0]}\to CH^1(B)
\end{equation}
in $\Ab[1/p]$ which coincides with $\langle, \rangle_{(l)}$ in $\Ab\otimes \Z_{(l)}$ for each $l$.\\
b) The identities of Propositions   \ref{r4} (see Remark \ref{r10} b)) and \ref{p5} (iv) extend to these pairings.
\end{thm}

\begin{proof} a) Suppose first that $k$ is the perfect closure of a field $k_0$ finitely generated over $\F_p$, and that $B=B_0\otimes_{k_0} k$ for some smooth $k_0$-variety $B_0$. Then $CH^1(B_0)$ is a finitely generated abelian group \cite{picfini}, and $CH^1(B_0) \otimes \Z[1/p]$ does not change under purely inseparable extensions; in particular,  $CH^1(B) \otimes \Z[1/p]$ has finite torsion and a fortiori verifies the hypothesis of Corollary \ref{t4.1}. The result then follows from this theorem and Lemma \ref{l4.2}.

In general, the situation is defined over such a subfield of $k$, so reduces to the first case.

b) Let $X_1,X_2$ be (proper) regular, and let $\gamma\in CH^{\dim X_2}(X_1\times_K X_2)$. We need to prove the analogue of \eqref{eq19a}:
\[\langle, \rangle_1\circ \gamma^*\otimes 1 = \langle, \rangle_2\circ 1\otimes\gamma_*\]
where $\langle, \rangle_i$ is the height pairing of $X_i$. By the uniqueness statement of Corollary \ref{t4.1}, it suffices to prove this identity after localising at $l$ for all $l\ne p$. Let $\pi_i:X_{i,l}\to X_i$ ($i=1,2$) be two admissible alterations of generic degrees $d_i$ prime to $l$, and let $\gamma_l=\pi_2^*\circ \gamma \circ (\pi_1)_*\in CH^{\dim X_2}(X_{1,l}\times_K X_{2,l})$, so that $d_2\gamma\circ (\pi_1)_*=(\pi_2)_*\gamma_l$ and $\gamma_l\circ \pi_1^*=d_1\pi_2^*\circ \gamma$.   By Lemma \ref{l4.2}, we have, with obvious notation:
\begin{multline*}
\langle, \rangle_1\circ \gamma^*\otimes 1 = d_1^{-1} \langle, \rangle_{1,l}\circ \pi_1^*\gamma^*\otimes \pi_1^*
= d_1^{-1}d_2^{-1} \langle, \rangle_{1,l}\circ \gamma_l^* \pi_2^*\otimes \pi_1^*\\
\overset{(a)}{=}d_1^{-1}d_2^{-1} \langle, \rangle_{2,l}\circ \pi_2^*\otimes (\gamma_l)_* \pi_1^*=d_2^{-1} \langle, \rangle_{2,l}\circ \pi_2^*\otimes \pi_2^* \gamma_*\\
=\langle, \rangle_2\circ 1\otimes  \gamma_*
\end{multline*}
where (a) used \eqref{eq19a} for $\gamma_l$.

The identity of Proposition \ref{p5} (iv) is extended in similar fashion.
\end{proof}

We shall use the following fact in the proof of Theorem \ref{t2}:

\begin{ex}\label{ex1} Suppose that $X$ is an abelian variety. For $a\in X(K)$, write $\tau_a$ for the translation by $a$. It yields a self-correspondence of degree $0$ still denoted by $\tau_a$, and we have the obvious formula ${}^t\tau_a=\tau_{-a}$. This yields the identity (see Remark \ref{r10} b))
\[\langle\tau_a^*\alpha,\beta\rangle =\langle\alpha,\tau_{-a}^*\beta\rangle\]
for $(\alpha,\beta)\in CH^i(X)^{[0]}\times CH^{d+1-i}(X)^{[0]}$.
\end{ex}

\begin{rque}\label{r8} The functoriality of Proposition \ref{p3.4} c) means that the subgroups $CH^i(X)^{[0]}$, for varying $X$ and $i$, define an \emph{adequate equivalence relation} on algebraic cycles with integral coefficients on smooth projective $K$-varieties. This adequate relation a priori depends on the choice of $B$, but see Conjecture \ref{c1} and Remark \ref{r6.1} below.
\end{rque}

\subsection{Extension to imperfect fields}\label{s3.1} Let $X,K,B$ be as in the introduction, but relax the assumption that $k$ is perfect; specifically, we assume $k$ imperfect of characteristic $p$. Write $k^p$ (resp. $K^p$, $B^p$, $X^p$ for the perfect closure of $k$ (resp. for $K\otimes_k k^p$, $B\otimes_k k^p$, $X\otimes_K K^p$). 

We define $CH^i(X)^{[0]}$ as the inverse image of $CH^i(X^p)^{[0]}$ under the pull-back morphism $CH^i(X)\to CH^i(X^p)$. We claim that the pairing \eqref{eq9} for $X^p$ induces a similar pairing for $X$, with the same properties.

Since the homomorphism $\lambda: CH^1(B)\to CH^i(B^p)$ has $p$-primary torsion kernel and cokernel, this is trivial if we accept to replace $CH^1(B)$ by $CH^1(B)\otimes\Z[1/p]$ (note that $\Ker \lambda$ and $\Coker\lambda$ do not have finite exponent, so $\lambda$ is not an isomorphism in $\Ab[1/p]$). We can avoid this, however, by observing that all constructions involved in constructing \eqref{eq9} for $X^p$ and proving its properties are defined over some finite subextension of $k^p/k$.

\section{Homologically and algebraically trivial cycles} \label{s4}

From now on, we write
\[CH^i(X)^{(0)}= \{\alpha\in CH^i(X)\mid \ \exists n\ne 0: n \alpha\in CH^i(X)^0\}\]
for the saturation of $CH^i(X)^0$. We have the inclusion
\begin{equation}\label{eq5.1}
CH^i(X)^{(0)}\subseteq CH^i_\num(X)
\end{equation}
by Lemma \ref{l4.3} d) and the fact that $CH^i(X)/CH^i_\num(X)$ is torsion-free.

\subsection{Conjectures}  The following is a numerical analogue to \cite[Conj. 2.2.5]{beilinson}.

\begin{conj}\label{c1} The inclusion \eqref{eq5.1} is an equality.
\end{conj}

Let the index $l$ denote homological equivalence for $l$-adic cohomology, $l\ne \car k$. Conjecture \ref{c1} implies

\begin{conj} \label{c2} One has the inclusion $CH^i_l(X)\allowbreak \subseteq CH^i(X)^{(0)}$.
\end{conj}

Conversely, Conjecture \ref{c2} implies Conjecture \ref{c1} under Gro\-then\-dieck's standard conjecture D, by Propositions \ref{p2.2} (and \ref{p3.4} a) in characteristic $p$).

\begin{prop}\label{l7a} Conjecture \ref{c2} is true if $X$ admits a model $f:\sX\to B$ with $f$ smooth. 
\end{prop}

\begin{proof} 
This follows from the smooth and proper base change theorem (see Remark \ref{r2.4}).
\end{proof}

\begin{rques}\label{r6.1} a) More generally, Proposition \ref{p2.1} shows that $CH^i(X)^0$ contains the image of $CH^i(\sX)^0_{\cB,l}$ for any model $f:\sX\to B$ of $X$ with $\sX$ smooth.\\
b) Suppose $X$ smooth (not just regular). For clarity, let us write $CH^i(X)^{(0)}_B$ to mark the dependence of $CH^i(X)^{(0)}$ on the model $B$. If $U$ is an open subset of $B$, we obviously have $CH^i(X)^{(0)}_B\subseteq CH^i(X)^{(0)}_U$, and this direct system is \emph{essentially constant} by Proposition \ref{p3.4} b). For $U$ small enough, Proposition \ref{l7a} thus gives inclusions
\[CH^i_l(X) \subseteq CH^i(X)^{(0)}_U\subseteq CH^i_\num(X)\]
where the middle group does not change when $U$ gets smaller (note that equality on the right is \emph{not} clear: see Remark \ref{r2.4}). In view of Remark \ref{r8}, this defines a new adequate equivalence on smooth projective $K$-varieties, this time independent of the choice of $B$ (and which conjecturally agrees with numerical equivalence). 
\end{rques}

\begin{thm}\label{t5.1} Conjecture \ref{c1} is true in the following cases:
\begin{thlist}
\item $i=1,d$.
\item $\car K=0$, $f$ is smooth and 
\begin{itemize}
\item either $i\in \{2,d-1\}$
\item or $X$ is ``of abelian type'' (i.e. its homological motive is isomorphic to a direct summand of the motive of an abelian variety).
\end{itemize}
\end{thlist}
\end{thm}

\begin{proof} For (i), see Proposition \ref{p2} b) below. For (ii), homological and numerical equivalences agree in the said cases by Lieberman \cite{lieb}. Therefore, the statement follows from Proposition \ref{l7a}.
\end{proof}

\subsection{Algebraic equivalence}

\begin{thm}\label{p2} a) One has $CH^i_\alg(X)\subseteq CH^i(X)^{(0)}$.\\
b) Conjecture \ref{c1} is true for $i=1,d$.
\end{thm}

Of course, b)  follows from a) (using Matsusaka's theorem \cite{mats} in case $i=1$).

To prove a), we first reduce to the case where $X$ has a smooth model $\sX$ as in Section \ref{s4a}: this is automatic if $\car k=0$ by Proposition \ref{p3.1}, and if $\car k>0$ we first reduce to $k$ perfect as in Subsection \ref{s3.1}, then we can use Proposition \ref{p3.3} a) and a transfer argument.

We now give ourselves a model $f:\sX\to B$ of $X$ with $\sX$ smooth. The proof is in two steps.

\subsubsection*{Step 1} $d=1$ and two sections $\tilde c_0,\tilde c_1$ of $f$ are given. Let $c_0,c_1$ be their generic fibres and $\alpha=[c_0]-[c_1]$.

\begin{lemma}\label{l5.1} There exists an integer $N>0$ such that $N\alpha\in CH^1(X)^0$.
\end{lemma}

\begin{proof} Let $\tilde \alpha= [\tilde c_0(B)]-[\tilde c_1(B)]\in CH^1(\sX)$. Then $j^*\tilde \alpha\in CH^1_\alg(X)\subseteq CH^1_\num(X)$. We now need to find $N>0$ and $\xi\in \Ker j^*$ such that $N\tilde \alpha+\xi\in CH^1(\sX)^0$.  We shall look for $\xi$ in the form
\[\xi=\sum_{b\in B^{(1)}} (\iota_b)_* \xi_b\]
where $\iota_b:\sX_{Z_b}\inj \sX$ is the inclusion (with $Z_b=\overline{\{b\}}$ as usual) and each $\xi_b$ is a linear combination of classes of irreducible $\delta$-dimensional components $\sX_{Z_b}^\lambda$ of $\sX_{Z_b}$ (almost all $\xi_b$ will be $0$). For this, I claim that the method of \cite[III.8]{silverman} extends to this case:

The first thing to check is that the hypothesis of loc. cit., Proposition III.8.3 is verified, namely that $\langle \tilde \alpha,[\sX_{Z_b}]\rangle_b=0$ for all $b\in B^{(1)}$. For simplicity, write $Z$ and $\iota$ instead of $Z_b$ and $\iota_b$. 
Up to removing a proper closed subset from $Z$, we may assume it smooth. In the Cartesian square of the diagram
\[\begin{CD}
&&Z@>g_i=(d_i,\iota')>> \sX_Z\times B\\
&&@Vd_i VV @V1\times \tilde c_i VV\\
Z @<f_Z<< \sX_Z@>\gamma_\iota>> \sX_Z\times \sX
\end{CD}\]
where $d_i=(\tilde c_i)_{|Z}$ and $\iota'$ is the inclusion $Z\inj B$, the top horizontal map $g_i$ is a regular embedding of codimension $\delta+1$ as the composite of the two regular embeddings
\[Z\by{\delta} Z\times Z\by{d_i\times \iota'}\sX_Z\times B.\]

Here we use that the embedding $d_i$ is regular (EGA IV$_4$, Prop. 19.1.1). Then
\begin{multline*}
\langle [\tilde c_i(B)],[\sX_Z]\rangle_b= (f_Z)_*\gamma_\iota^!([\sX_Z]\times (\tilde c_i)_*[B])
= (f_Z)_*\gamma_\iota^!((1\times \tilde c_i)_*[\sX_Z\times B])\\
\overset{(a)}{=} (f_Z)_*(d_i)_*\gamma_\iota^![\sX_Z\times B] = \gamma_\iota^![\sX_Z\times B]\\
\overset{(b)}{=}g_i^![\sX_Z^\lambda\times B]\overset{(c)}{=}\delta^!(d_i^![\sX_Z]\times {\iota'}^![B])
=[Z]
\end{multline*}
where (a) (resp. (b), (c)) is once again Lemma \ref{l2.5} a) (resp. b), c)).

Now 
\[CH_{\delta+i-1}(\sX_Z)=CH_\delta(\sX_Z)\osi \bigoplus_\lambda \Z [\sX_Z^\lambda]\] 
where the $\sX_Z^\lambda$ are the irreducible components of $\sX_Z$ of dimension $\delta$: this follows from \cite[Ex. 1.8.1]{fulton} by induction on the number of components. The second thing to observe is that the statement and proof of \cite[Prop. III.8.2]{silverman} apply verbatim, namely that the quadratic form $\alpha \mapsto \langle \iota_*\alpha, \alpha\rangle_b$ on $CH_{\delta}(\sX_Z)$ is negative, with kernel generated by $[\sX_Z]$. Indeed, this is a local computation so we can consider the fibre of $\sX$ over $\Spec \sO_{B,b}$ and simply apply the said proposition. (The fact that $f_*\sO_\sX=\sO_B$, which is used in its proof, follows from the fact that $X$ is geometrically connected since it has rational points, and that $B$ is normal.)

We can now find $N$ and $\xi$ just as in \cite[Prop. III.8.3]{silverman}.
\end{proof}

\subsubsection*{Step 2}The general case. Let $\alpha \in CH^i_\alg(X)$. By \cite[Lemma 3.8]{achteretal}, there exist an integer $s\ge 0$, a smooth projective $K$-curve  $C$, two rational points $c_0,c_1\in C(K)$ and an element $y\in CH^i(C\times X)$ such that $p^s\alpha =(c_0^* - c_1^*)y$ (recall that $p$ is the exponential characteristic of $k$). 

\begin{lemma}[Q. Liu, cf. \protect{\cite{liu-tong}}]\label{l-liu} There exists a closed subset $F\subset B$ of codimension $> 1$ such that $C$ lifts to a regular proper $(B - F)$-scheme $\sC$ and the $c_i$'s lift to sections $\tilde c_i$ of $\sC\to B - F$.
\end{lemma}

\begin{proof} See Proposition \ref{pa} of the appendix.
\end{proof}

By Step 1 and Lemma \ref{l-liu}, there exists an integer $N>0$ such that $N([c_1]-[c_1])\in CH^1(C)^0$. Then  $Np^s\alpha =N(c_0^* - c_1^*)y= y^*N([c_0]-[c_1]) \in CH^i(X)^0$ by Proposition \ref{l6}, where $y$ is considered as a correspondence (Weil-Bloch trick). This concludes the proof of Theorem \ref{p2}.

\begin{rque}\label{r9} There is a statement parallel to Theorem \ref{p2} in \cite[Lemma 2.2.2 b)]{beilinson}, with a similar proof.
\end{rque}

\subsection{Example: elliptic curves}\label{s5.3} In Step 1 of the proof of Theorem \ref{p2}, suppose $\delta=1$, $B$ projective and that $X$ is an elliptic curve. Applying $\deg:CH^1(B)\to \Z$, we get an integral pairing $\langle,\rangle$ on the finite index subgroup $CH^1(X)^0\cap X(K)$ of $CH^1_\alg(X)=\Pic^0(X)=X(K)$; this pairing coincides with the Néron-Tate height pairing by the description in \cite[Th. III.9.3]{silverman}. 

\begin{thm}\label{t5.2} Assume $k$ algebraically closed. Then\\
a) $CH^1(X)^0\cap X(K)$ contains the subset denoted by $X(K)_0$ in \cite[Rem. III.9.4.2]{silverman}.\\ 
b) If $\sX$ is a minimal model, $X(K)_0$ is a subgroup and the pairing
\begin{equation}\label{eq5.11}
X(K)_0\times X(K)_0\to \Pic(B)
\end{equation}
of loc. cit., Th. III.9.5 (b) equals $-\langle,\rangle$.
\end{thm}

\begin{proof} a) Let $P\in X(K)$. As in  Lemma \ref{l5.1}, write $\tilde P:B\to \sX$ for the section of $f$ extending $P$ (here its existence as a morphism is automatic since $\delta=1$, by the valuative criterion of properness). What is written $[\tilde P(B)]=\tilde P_*[B]$ in its proof of is denoted by $(P)$ in \cite{silverman}.  Since
\[(P)\cdot [\sX_b]=\tilde P_*[B]\cdot f^*b=[B]\cdot \tilde P^*f^*b=\deg(b)=1\]
(projection formula), and the intersection numbers of $(P)$ with the components of $\sX_b$ are $\ge 0$, this implies that $P$ meets exactly one component of $\sX_b$, with multiplicity $1$. 

By definition, $X(K)_0$ is the set of $P$'s such that $(P)$ meets the same component of $\sX_b$ as $(0)$ for all $b\in B^{(1)}$. Equivalently, $\deg(((P)-(0))\cdot [\sX_b^\lambda])=0$ for all $b$ and all such components. By \eqref{eq6}, this degree is none else than $\langle (P)-(0),\sX_b^\lambda\rangle_b$, so we get that $P\in X(K)_0$ $\Rightarrow$ $(P)-(0)\in CH^1(\sX)^0$ $\Rightarrow$ $P-0\in CH^1(X)^0$.

b)  What we use here is that 
\begin{equation}\label{eq5.12}
\tilde P =\tau_P\circ \tilde 0
\end{equation}
for all $P\in X(K)$, where $\tau_P$ is the translation by $P$ \cite[Prop. III.9.1]{silverman}. This already implies that $X(K)_0$ is a subgroup of $X(K)$.

We start with a convenient description of \eqref{eq5.11} by reformulating part (a) of \cite[Th. III.9.5]{silverman}. For $P,Q\in X(K)$, we have $j^*((P+Q)-(P)-(Q)+(0))=0$ in $\Pic^0(X)$; the sequence
\[0\to \Pic(B)\by{f^*} \Pic(\sX)\by{j^*} \Pic(X)\to 0\]
is exact except at $\Pic(\sX)$ where its homology is given by $\bigoplus_b CH_1(\sX_b)/[\sX_b]$ (see \cite[3.2 a)]{langneron}). If now $P,Q\in X(K)_0$, then
\begin{multline*}
(P+Q)-(P)-(Q)+(0)= ((P+Q)-(0)) -((P)-(0)) -((Q)-(0))\\
\in CH^1(\sX)^0
\end{multline*}
which implies that its homology class is $0$ by the non-degeneracy of the intersection pairings on $CH_1(\sX_b)/[\sX_b]$. Thus $(P+Q)-(P)-(Q)+(0)=f^*[P,Q]$ for a unique $[P,Q]\in \Pic(B)$. In particular,
\begin{equation}\label{eq5.13}
[P,Q]=\tilde R^*((P+Q)-(P)-(Q)+(0)) \quad \forall\ R\in X(K).
\end{equation}

For convenience, we now write
\[ P*Q = f_*((P)\cdot (Q))\in \Pic(B)\]
for $P,Q\in X(K)$.

\begin{lemma} We have the identities $P*Q =\tilde Q^*(P)$ and $P*Q = 0*(P-Q)$.
\end{lemma}

\begin{proof} For the first identity,
\[f_*((P)\cdot (Q))=f_*(\tilde P_*[B]\cdot \tilde Q_*[P])=f_*(\tilde Q_*\tilde Q^*\tilde P_*[B])=\tilde Q^*(P) \]
by the projection formula. For the second one,
\begin{multline*}
\tilde Q^*(P)=\tilde 0^*\tau_Q^*(\tau_P)_*\tilde 0_*[B] =\tilde 0^*(\tau_Q)_*^{-1}(\tau_P)_*\tilde 0_*[B]\\
=\tilde 0^*(\tau_{P-Q})_*\tilde 0_*[B]=\tilde 0^*(P-Q).
\end{multline*}
\end{proof}

\begin{rque} Since $P*Q=Q*P$, we also get the intriguing identity $0*P = 0*(-P)$.
\end{rque}

To prove the claim of Theorem \ref{t5.2} b), we now apply \eqref{eq5.13} with $R=Q$:
\begin{multline*}
[P,Q]=\tilde Q^*((P+Q)-(P)-(Q)+(0))\\
= (P+Q)*Q -P*Q -Q*Q +0*Q\\
= 0*P  -P*Q-0*0 +0*Q\\
= P*0  -P*Q-0*0 +0*Q=-\langle P,Q\rangle.
\end{multline*}
\end{proof}

\begin{rques}\label{r5.1} a) We have $X(K)_0=\sN^0(B)$, where $\sN^0$ is the identity component of the Néron model $\sN$ of $X$. Indeed, $\sN$ is isomorphic to the smooth locus $\sX_\sm$ of $\sX$ \cite[Prop. 1.15]{artin} and $\sN^0$ contains the $0$-section $(0)$. For any $P\in X(K)$, $(P)\subset \sX_\sm$ (see end of Step 2 in the proof of Proposition \ref{pa}), and  $(P)\in \sN^0$ if and only if $P\in X(K)_0$ since $\sN^0_b$ is the identity component of $\sN_b$ for all $b\in B^{(1)}$ by definition of $\sN^0$. \\
b) Suppose that $P-0\in CH^1(X)^0$. We can find a fibral divisor $\xi$ such that $(P)-(0)-\xi$ is orthogonal to all fibral divisors (as in the proof of Lemma \ref{l5.1}, with $N=1$), and this divisor is unique modulo $\IM f^*$ by \cite[Prop. III.8.3]{silverman}. By loc. cit., Lemma III.9.4 (or by a)), $X(K)/X(K)_0=X(K)/\sN^0(B)$ is finite, so the class of $\xi$ is torsion in each $CH_1(\sX_b)/[\sX_b]$. Thus $CH^1(X)^0/\sN^0(B)\inj \bigoplus_{b\in B^{(1)}} (CH_1(\sX_b)/[\sX_b])_\tors$.
\end{rques}

\section{The pairing in codimension $1$}\label{s5}

In this section, we assume $X$ projective and geometrically irreducible. 
Recall that $\delta=\trdeg(K/k)=\dim B$. We shall study the height pairing \eqref{eq9} for $i=1$, in $\Ab\otimes\Q$; note that $CH^i(X)^{(0)}=CH^i_\num(X)$ for $i=1,d$ by  Proposition \ref{p2}.

\subsection{A general result} We write $T(X)\subset CH^{d}_\num(X)=CH^{d}(X)_0$ for the Albanese kernel. For an abelian $K$-variety $A$, write $\Tr_{K/k} A$ for its $K/k$-trace and 
\[\LN(A,K/k) = A(K)/(\Tr_{K/k} A)(k)\]
for its \emph{Lang-Néron group}: it is finitely generated by the Lang-Néron theorem \cite{LN}. We shall need the following classical fact:

\begin{lemma}\label{l5.2} The Albanese map $a_X:CH^d(X)_0\to \Alb_X(K)$ has a cokernel of finite exponent.
\end{lemma}

\begin{proof} This could be deduced from \cite[Prop. A.1]{griff}; here is a different and more direct proof. Choose a smooth irreducible multiple hyperplane section of dimension $1$ $i:C\inj X$. By the usual transfer argument, we may assume that $X$ has a rational point lying on $C$. Then $a_C$ is bijective. By \cite[Lemma 2.3]{murre}, the composition
\begin{equation}\label{eq5.10}
\Pic^0_X\by{i^*}\Pic^0_C=\Alb_C\by{i_*} \Alb_X
\end{equation}
is an isogeny, hence $\Coker i_*(K)$ has finite exponent and so does its quotient $\Coker a_X$.
\end{proof}

\begin{thm}\label{t2} a) The pairing $\langle, \rangle$ vanishes on  $CH^1_\num(X)\times T(X)$. \\
b) This induces a pairing (in $\Ab\otimes \Q$)
\[\langle, \rangle:\Pic^0(X)\times \Alb_X(K)\to CH^1(B).\]
c) Suppose $B$ projective. Composing this pairing with the projection $CH^1(B)\allowbreak\to N^1(B)$ (where $N^1(B)$ is the group of cycles of codimension $1$ modulo numerical equivalence) induces a pairing
\begin{equation}\label{eq5.2}
\langle, \rangle_\num:\LN(\Pic^0_X,K/k)\times \LN(\Alb_X,K/k)\to N^1(B).
\end{equation}
\end{thm}

\begin{proof} a) Up to extending scalars to the perfect closure of $k$, we may assume $k$ perfect. Let $L/K$ be a finite extension. Let $B_L$ be the normalisation of $B$ in $L$; up to removing from $B$ a closed subset $F$ of codimension $\ge 2$ and from $B_L$ the inverse image of $F$ (which does not affect $CH^1(B)$ or $CH^1(B_L)$), we may assume $B_L$ smooth. In $\Ab[1/p]$, the map $CH^1(B)\to CH^1(B_L)$ is a monomorphism (transfer argument). In view of  the functoriality in Theorem \ref{p3} b), to prove the vanishing we may thus increase scalars as much as we wish. In particular, we may assume that $X(K)\neq \emptyset$.

Let $x\in X(K)$ and let $a:X\to \Alb_X$ be the corresponding Albanese map. Then $a$ induces an isomorphism $a^*:\Pic^0(\Alb_X)\iso \Pic^0(X)$, and $a_*:CH_0(X)\to CH_0(\Alb_X)$ sends $T(X)$ into $T(\Alb_X)$. Still by functoriality, we are reduced to the case $X=\Alb_X=:A$.

The sequel is inspired by Néron's proof of \cite[Prop. 7]{neron}. In order to reason with elements, pick a representative of $\langle, \rangle$ in $\Ab$ as in Remark \ref{r10} a). Let $\beta\in T(A)$, and let $\bar K$ be an algebraic closure of $K$. In $T(A_{\bar K})$, we may write $\beta_{\bar K}= \sum_i ([a_i+b_i]-[a_i]-[b_i]+[0])$, with $a_i,b_i\in A(\bar K)$. Choose $L/K$ finite such that all $a_i$'s are rational over $L$. As above, we may extend scalars from $K$ to $L$, and thus reduce to $\beta=[a+b]-[a]-[b]+[0]$ for $a,b\in A(K)$. The vanishing now follows from Example \ref{ex1} and the theorem of the square \cite[II.6, Cor. 4]{mumford}.

b) follows immediately from a) and  Lemma \ref{l5.2}, which implies that \allowbreak $CH^d(X)_0/T(X)\to \Alb_X(K)$ is an isomorphism in $\Ab\otimes\Q$. 

c) We may assume $k$ algebraically closed; then the claim follows from the divisibility of $Y(k)$ for an abelian $k$-variety $Y$ and the finite generation of $N^1(B)$.
\end{proof}

\subsection{Another conjecture} For the needs of Theorem \ref{t3} below, we introduce a new conjecture. From now on, $B$ is projective as in Theorem \ref{t2} c).

Let $R$ be a discrete valuation ring with quotient field $K$ and residue field $E$. Suppose that an abelian $K$-variety  $A$ has good reduction with respect to $R$; then its Néron model $\sA$ is an abelian scheme over $\Spec R$, whose special fibre $A_s$ is an abelian $E$-variety. We have a specialisation homomorphism
\begin{equation}\label{eq5.8}
A(K)=\sA(R)\to A_s(E).
\end{equation}

Suppose now that $R$ contains $k$. The notion of $K/k$-trace readily extends to a notion of $R/k$-trace for abelian $R$-schemes; viewing these traces as right adjoints shows that
\begin{itemize}
\item $\Tr_{R/k} \sA$ exists and equals $\Tr_{K/k} A$;
\item the `special fibre' functor yields a canonical morphism $\Tr_{K/k} A\to 
\Tr_{E/k} A_s$. 
\end{itemize}

It follows that \eqref{eq5.8} induces a homomorphism of Lang-Néron groups
\begin{equation}\label{eq5.9}
\LN(A,K/k)\to \LN(A_s,E/k).
\end{equation}

\begin{conj}\label{co4} Assume that $A$ has semi-stable reduction at every point of $B^{(1)}$, and that $\delta>1$. For any projective embedding $B\inj \P^N$, there exists a smooth, geometrically connected hyperplane section $h$ of $B$ such that $A$ has good reduction at $h$ and the kernel of \eqref{eq5.9} is finite, with  $E=k(h)$.
\end{conj}

Suppose $A$ constant. Then \eqref{eq5.9} may be rewritten as
\[
\Hom_k(\Alb_B,A)\to \Hom_k(\Alb_h,A),
\]
and Conjecture \ref{co4} follows from the surjectivity of $\Alb_h\to \Alb_B$ (see \eqref{eq5.10}). This gives some evidence for this conjecture.

\begin{rque} Perhaps the hypotheses of Conjecture \ref{co4} are too weak. In any case, we only need it in the special case $A=\Pic^0_X$, when $X$ satisfies the conclusion of Lemma \ref{ldJ} (or any suitable variant of it); it may be easier to prove in such a case.
\end{rque}

\subsection{A technical lemma} This lemma will be needed in the proofs of Theorem \ref{t3} and Proposition \ref{p5.4} below.

\begin{lemma}\label{ldJ} Suppose that $d=1$. Then there exists an alteration $\tilde B\to B$, with $\tilde B$ smooth, such that $X\otimes_K k(\tilde B)$ has a projective model $f:\sX\to \tilde B$ where $\sX$ is smooth over $k$ and, for all $b\in \tilde B^{(1)}$, the irreducible components of $\sX_b$ are smooth over $k(b)$.
\end{lemma}

\begin{proof} Start from a projective  embedding $X\inj \P^N_K$ and consider its closure $\sX_0$ in $\P^N_B$. In the following reasoning using results of \cite{dJ2}, we always take the group $G$ appearing there equal to $1$. By \cite[Th. 5.9]{dJ2} (or just \cite[Th. 2.4 and Lemma 5.7]{dJ2}), we may (projectively) alter $f_0:\sX_0\to B$ into $f_1:\sX_1\to B_1$  so that $f_1$ is a projective quasi-split semi-stable curve in the sense of \cite[\S\ after Lemma 5.6]{dJ2}. This condition is stable under base change, hence, by the reasoning at the end of the proof of \cite[Th. 5.13]{dJ},  we may alter $B_1$ into $B_2$ so that $B_2$ is smooth and $f_2:\sX_2:=\sX_1\times_{B_1} B_2\to B_2$ verifies the hypotheses of \cite[Prop. 5.11]{dJ2} (note that varieties over a field verify \cite[(5.12.1)]{dJ2} by \cite[Th. 4.1]{dJ}); in particular, $\tilde B:=B_2$ is smooth.  Next, the beginning of the proof of  \cite[Prop. 5.11]{dJ2} yields a modification $\pi:\sX_3\to \sX_2$ such that the singular locus $\Sigma$ of $\sX_3$ is smooth of codimension $\ge 3$ and $f_3:\sX_3\to \tilde B$ is still a quasi-split semi-stable curve. The end of this proof then yields a desingularisation $\sX_4$ of $\sX_3$ by blowing up the components of $\Sigma$. Since they lie over points of codimension $\ge 2$ in $\tilde B$, this does not affect the fibres of $f_3$ at points of codimension $1$, so $f_4:\sX_4\to \tilde B$ is ``quasi-split semi-stable in codimension $1$''.

We are left to desingularise the singular components of $(\sX_4)_b$ for all $b\in \tilde B^{(1)}$. Let $D$ be such a component, and let $x$ be a singular point of $D$. Note that $x$ does not lie on any other component, since all singular points of $(\sX_4)_b$ are quadratic by the ``semi-stable'' condition. By the ``quasi-split'' one, the completion of $\sO_{\sX_4,x}$ is isomorphic to $k[[u,v,t_1,\dots, t_\delta]]/(uv-t_1)$, where $t_1$ is a local equation of $D$ (compare \cite[2.16]{dJ}). The ideal of $x$ is $(u,v,t_1)$. Blowing up this ideal retains the regularity of $\sX_4$, separates the two branches of $D$ at $x$ (making its strict transform regular at the two corresponding points) and adds a smooth irreducible exceptional divisor. We have therefore decreased by $1$ the total number of singular points of the irreducible components of $(\sX_4)_b$. Since only finitely many $b$'s are involved, we end the process after a finite number of steps.
\end{proof}

\subsection{A negativity theorem}

\begin{thm}\label{t3} Let $L\in \Pic(X)$ and $\ell\in \Pic(B)-\{0\}$. Consider the quadratic form   
\[q=q(X,B,L,\ell):\LN(\Pic^0_X,K/k)\ni \alpha\mapsto \deg\left(\langle \alpha,L^{d-1}\alpha\rangle_\num\cdot \ell^{\delta-1}\right)\]
obtained from the pairing of Theorem \ref{t2} c). If $L$ is ample and $\delta=1$ (hence $\ell^{\delta-1}=1$),  then $q(X,B,L,\ell)$ is negative definite (in particular, non-degenerate).  If  Conjecture \ref{co4} holds for $\Pic^0_X$ when $d=1$ and in the situation of Lemma \ref{ldJ}, this extends to $\delta>1$ for $\ell$ ample. 
\end{thm}

\begin{rque} As pointed out in Remark \ref{r10} b), the notation using elements  is abusive in  $\Ab\otimes \Q$. Theorem \ref{t3} could  be converted into an arrow-theoretic statement; similarly, the notion ``negative definite'' for a quadratic form with values in $\Z$ is unambiguous in $\Ab\otimes \Q$, by using Remark \ref{r10} a).

However, converting the proof below into arrow-theoretic notation would be cumbersome at best. Since the source and target of the quadratic form $q$ are finitely generated abelian groups, we can tensor everything with $\Q$ (i.e. apply the natural functor from $\Ab\otimes \Q$ to $\Q$-vector spaces) without losing information, and reason with honest elements. This is what we do in this proof.
\end{rque}

\begin{proof}  a) We first reduce to $d=1$ as follows. Suppose $d>1$. We may assume $L$ very ample. Let $i:C\inj X$ be a  smooth irreducible curve given by successive hyperplane sections from the projective embedding determined by $L$. By the functoriality of Theorem \ref{p3}, we have
\[\langle i^*\alpha,i^*\alpha\rangle_\num = \langle \alpha,i_*i^*\alpha\rangle_\num = \langle \alpha,L^{d-1}\cdot \alpha\rangle_\num,\]
hence $q(X,B,L,\ell)(\alpha) = q(C,B,i^*L,\ell)(i^*\alpha)$.  By the isogeny \eqref{eq5.10}, \allowbreak $\LN(\Pic^0_X,K/k)\to \LN(\Pic^0_C,K/k)$ is mono in $\Ab\otimes \Q$.

We now assume $d=1$. 

b) We reduce to the situation of Lemma \ref{ldJ}. Let $\sX\by{f}\tilde B$ be as in loc. cit.  Since $\pi:\tilde B\to B$ is projective, pick a very ample divisor $\sL$ relative to $\pi$. By EGA II, Prop. 4.4.10 (ii), $\sL+ n\pi^*\ell$ is then very ample (relative to $\tilde B\to \Spec k$) for all $n\gg 0$. Let $\alpha\in \LN(\Pic^0_X,K/k)-\{0\}$. Assuming the theorem true over $\tilde B$, we have
\[\deg\left(\langle \pi^*\alpha,\pi^*L^{d-1}\pi^*\alpha\rangle_\num\cdot (\sL+ n\pi^*\ell)^{\delta-1}\right)<0\]
for all $n\gg 0$. This is a polynomial in $n$, with dominant term 
\[
\deg\left(\langle \pi^*\alpha,\pi^*L^{d-1}\pi^*\alpha\rangle_\num\cdot \pi^*\ell^{\delta-1}\right)
=\deg\left(\langle \alpha,L^{d-1}\alpha\rangle_\num\cdot \ell^{\delta-1}\right)
\]
by \eqref{eq20}, which must be negative.

We now assume that we are in the situation of Lemma \ref{ldJ}.

c) Assume $\delta=1$. Observe that the pairing \eqref{eq2}, composed with the degree, is then the intersection pairing. By the Hodge index theorem, this pairing has signature $(1,\rho-1)$ where $\rho=\rk N^1(\sX)$. Since  $N^1(\sX)^0$ is the orthogonal of the  isotropic vector $f^*t$ for $t\in N^1(B)-\{0\}$,  the restriction of the intersection pairing to this subspace is negative with kernel generated by $f^*t$. Since $f^*t$ also generates the kernel of $N^1(\sX)^0\to \LN(\Pic^0_X,K/k)$, the quadratic form $q$ is negative definite, as requested.

d) Assume finally $\delta>1$. Similarly to a), we may assume $\ell$ very ample. We may also assume  $k$ algebraically closed (in particular, infinite). Let $Z\subset B$ be the locus of non-smoothness of $f$. In the family of hyperplane sections of $B$ relative to the projective embedding given by $\ell$, only finitely many may be contained in $Z$, therefore we can pick a smooth hyperplane section $h\not\subset Z$. By induction, there exists a smooth ample curve $i:\Gamma\subset B$ determined by $\ell$ such that the generic fibre $X(E)$ of $\sX_\Gamma=f^{-1}(\Gamma)$ is smooth over  $E=k(\Gamma)$.  

Write $I:\sX_\Gamma\inj\sX$, $g:\sX_\Gamma\to \Gamma$ for the two corresponding projections.  For $\tilde \alpha\in CH^1(\sX)$, we have
\[\langle \tilde \alpha,\tilde \alpha\rangle \cdot \ell^{\delta-1}=i_*i^*f_*(\tilde \alpha^2)=i_*g_*I^!(\tilde \alpha^2). \]

Since $\deg_B\circ i_*=\deg_\Gamma$, it is enough to compute $g_*I^!(\tilde \alpha^2)$.

Choose a resolution of singularities $\pi:\sY\to \sX_\Gamma$ of the surface $\sX_\Gamma$; let $\tilde I=I\circ \pi$  and $\tilde g=g\circ \pi$. The same reasoning as in the proof of Proposition \ref{p5} (iv) yields the identity $I^!=\pi_*\tilde I^*$, hence
\[g_*I^!(\tilde \alpha^2)=\tilde g_*\tilde I^*(\tilde \alpha^2)=\tilde g_* (\tilde I^*\tilde\alpha)^2. \]

Now there exists a finite extension $E'/E$ with smooth projective $k$-model $\Gamma'$, and a semi-stable model $\sY'$ of $X(E)\otimes_E E'$ over $\Gamma'$ mapping to $\sY$ by a morphism $\phi$. If $d=[E':E]$, we therefore have
\[(\tilde g\circ \phi)_* (\tilde I\circ\phi)^*(\tilde \alpha)^2= d\tilde g_* (\tilde I^*\tilde\alpha)^2. \]
Under Conjecture \ref{co4}, $\Gamma$ may be chosen such that the map induced by $\tilde I^*$ 
\[\LN(\Pic^0_X,K/k)\to \LN(\Pic^0_{X(E)},E/k)\]
has finite kernel, and our reduction to $\delta=1$ is complete.
\end{proof}

\subsection{Another pairing}\label{s5.1} Here we assume $B$ projective; we write $A=\Tr_{K/k} \Pic^0_X$ and $P=\Pic^0_B$. 

\begin{prop}\label{p5.4} Suppose $d=1$. In the pairing of Theorem \ref{t2} b), we have $\langle A(k),A(k)\rangle\subseteq \Pic^0(B)\{p\}$ in $\Ab\otimes \Q$, where $p$ is the exponential characteristic of $k$.
This induces a pairing in $\Ab\otimes \Q$
\begin{equation}\label{eq5.3}
\LN(\Pic^0_X,K/k)\times A(k)\to P(k)/P(k)\{p\}.
\end{equation}
\end{prop}

\begin{proof} We may first reduce to $k$ perfect and then pass to a finite extension of $K$, hence reduce to the existence of a smooth model $\sX$ (e.g. as in Proposition \ref{ldJ}).  By \cite[3.2 a)]{langneron}, we have
\[j^*\Pic^0(\sX)= A(k).\]

By Proposition \ref{p5.1}, $\langle \Pic^0(\sX),\Pic^0(\sX)\rangle$ is $p$-primary torsion, hence the claim.
\end{proof}

\begin{qn} Does \eqref{eq5.3} extend to arbitrary $d$, replacing $A(k)$ by $\Tr_{K/k} \Alb_X(k)$?
\end{qn}

Let $E=k(A)$. 
Using \cite[Th. 3.1]{milne},  we  
deduce from \eqref{eq5.3} a pairing
\[
\LN(\Pic^0_X,K/k)\times \Mor_k(A,A)\to \Mor_k(A,P)/\Mor_k(A,P)\{p\}).
\]

Evaluating on the identity $1_A\in \Mor_k(A,A)$, we get a homomorphism
\[\LN(\Pic^0_X,K/k)\to \Mor_k(A,P)/\Mor_k(A,P)\{p\})
\]
and using the canonical isomorphism $\Mor_k(A,P)\simeq P(k)\oplus \Hom_k(A,P)$, a final homomorphism
\begin{equation}\label{eq5.4}
\LN(\Pic^0_X,K/k)\to \Hom(\Tr_{K/k} \Pic^0_X, \Pic^0_B)
\end{equation}
because the right hand group is torsion-free. It is an exercise to check that, evaluating this homomorphism on $k$-points, we get back \eqref{eq5.3} (improved).

If $B=\P^1$ or $\Tr_{K/k} \Pic^0_X=0$, the right hand side is $0$ while the left hand side is nonzero in general. Yet we may ask:

\begin{qn}\label{q3} When is \eqref{eq5.4} surjective (in $\Ab\otimes \Q$)?
\end{qn}

\appendix

\section{Extending rational points to sections} 

\hfill by Qing Liu
\bigskip

\begin{prop}\label{pa} Let $B$ be a noetherian connected regular excellent scheme. Let $C$ be a
connected projective 
regular curve over the function field $K$ of $B$. Let
$c_1, \dots, c_n\in C(K)$. Then there exist an  
open subset $U\subseteq B$ with $\codim(B\setminus U, B)\ge 2$ 
and a proper scheme $\mathcal C \to U$, with $\sC$ regular, containing the $c_i$'s such that 
the latter extend to sections of $\mathcal C\to U$.
\end{prop}

{\bf Step 1.} We extend $C$ to a projective regular scheme $\mathcal C_0$ 
over some ``big'' open subset $U_0$ of $B$. 
\smallskip 

First we extend $C$ to an integral projective scheme $f : \mathcal X\to B$
(taking for instance the scheme-theoretical closure of $C$ 
in a suitable $\mathbb P^n_B$). Let $\mathcal X_{\mathrm{sing}}$ be
the closed subset of the singular points of $\mathcal X$. 
Then $V:=B\setminus f(\mathcal X_{\mathrm{sing}})$ is a dense open
subset of $B$ such that $\mathcal X_V$ is regular. 

Let $b_1, \dots, b_m$ be the codimension $1$ points
of $B$ inside of $B\setminus V$. We are going to extend $\mathcal X_V$ above an 
open subset $U_0$ of $B$ containing the $b_j$'s. 
For each $j\le m$, we have a relative integral curve 
$\mathcal X\times_B \Spec \cO_{B,b_j}$ over the discrete valuation 
ring $\cO_{B, b_j}$ with regular generic fiber $C$. As $B$ is
excellent,  
there exists a resolution of singularities 
$$\mathcal X'_j \to   \mathcal X\times_B \Spec \cO_{B,b_j} \to  \Spec
\cO_{B,b_j}.$$ 
Each $\mathcal X'_j$ is a projective regular curve over $ \Spec \cO_{B,b_j}$
and extends to a projective regular curve $\mathcal X_j$ 
over some open neighborhood $V_j\ni b_j$.  
Shrinking the (finitely many) $V_j$'s if necessary, we can 
suppose that for all 
$j, \ell\le m$, $\mathcal X_j$ and $\mathcal X_\ell$ coincide over
$V_{j}\cap V_\ell$ 
and that $\mathcal X_j$ coincides with $\mathcal X_V$ over $V\cap V_j$. Let 
$U_0$ be the union of $V$ and the $V_j$'s and let 
$\mathcal C_0\to U_0$ 
be obtained by glueing $\mathcal X_V$ and the $\mathcal X_j$'s. 
Then $\mathcal C_0$ is regular, proper over $U_0$ (by the fpqc descent
$V\coprod (\coprod_{1\le i\le m} V_j) \to U_0$ of properness, see EGA
IV$_2$, Proposition 2.7.1(vii)), and $\codim(B\setminus U_0, B)\ge 2$. 
\medskip

{\bf Step 2.} For all $i\le n$ we let  $P_i\subseteq \mathcal C_0$ be
the scheme-theoretical
closure of $\{ c_i\}$. Then $P_i\to U_0$ is proper birational, hence is an
isomorphism away from a closed subset $Z_i\subset U_0$ of codimension at
least $2$. To finish we let $U:=U_0\setminus (\cup_{1\le i\le n} Z_i)$ and let 
 $\mathcal C=(\mathcal C_0)_U\to U$. 
(As $U$ and $\mathcal C$ are regular, the section 
$(P_i)_U$ of $(\mathcal{C}\to U$ is contained in the smooth locus of $\mathcal C$ \cite[3.1, Prop. 2 and following paragraph]{BLR}.)

\enlargethispage*{40pt}

\end{document}